\documentclass{amsart}
\usepackage[toc,page]{appendix}
\usepackage{geometry}
\usepackage[all,cmtip]{xy}
\usepackage{amsmath}
\usepackage{amsfonts}
\usepackage{amssymb}
\usepackage{mathrsfs}   
\usepackage{amsthm}
\usepackage{xcolor}
\usepackage{cite}
\usepackage{stmaryrd}
\usepackage{graphicx}
\usepackage{pgf}
\usepackage{eepic}
\usepackage{pstricks}
\usepackage[all,cmtip]{xy}
\usepackage{epsfig}
\usepackage{fancyhdr}
\usepackage[backref=page]{hyperref}
\renewcommand*{\backref}[1]{}
\renewcommand*{\backrefalt}[4]{({%
    \ifcase #1 Not cited.%
          \or page~#2%
          \else pages #2%
    \fi%
    })}
    
 \usepackage{mathptmx}  

 \newcommand\imCMsym[4][\mathord]{%
  \DeclareFontFamily{U} {#2}{}
  \DeclareFontShape{U}{#2}{m}{n}{
    <-6> #25
    <6-7> #26
    <7-8> #27
    <8-9> #28
    <9-10> #29
    <10-12> #210
    <12-> #212}{}
  \DeclareSymbolFont{CM#2} {U} {#2}{m}{n}
  \DeclareMathSymbol{#4}{#1}{CM#2}{#3}
}
\newcommand\alsoimCMsym[4][\mathord]{\DeclareMathSymbol{#4}{#1}{CM#2}{#3}}

\imCMsym{cmmi}{124}{\CMjmath}
\imCMsym[\mathop]{cmsy}{113}{\CMamalg}
\imCMsym[\mathop]{cmex}{96}{\CMcoprod}
\alsoimCMsym[\mathop]{cmex}{97}{\CMbigcoprod}
 
\theoremstyle{plain}

\newtheorem{theorem}{Theorem}[section]

\newtheorem{proposition}[theorem]{Proposition}
\newtheorem{corollary}[theorem]{Corollary}

\newtheorem{lemma}[theorem]{Lemma}

\theoremstyle{definition}

\newtheorem{definition}[theorem]{Definition}

\theoremstyle{remark}
\newtheorem{remark}[theorem]{Remark}

\newcommand{\Q}{{\mathbb Q}}
\newcommand{\R}{{\mathbb R}}
\newcommand{\C}{{\mathbb C}}

\newcommand{\F}{{\mathbb F}}

\newcommand{\A}{{\mathbb A}}

\newcommand{\an}{^\mathrm{an}}

\newcommand{\spec}[1]{\mathrm{Spec}\left(#1\right)}

\newcommand{\norm}[1]{\left\vert#1\right\vert}

\newcommand{\bu}{\bullet}

\title{Rigid rational homotopy types}
\author{Christopher Lazda}
 \address{Dipartimento di Matematica Pura e Applicata \\
        Torre Archimede, Via Trieste, 63 \\ 
        35121 Padova \\ 
        Italia}
       \email{lazda@math.unipd.it}

\newcommand{\ho}[1]{\mathrm{Ho}(#1)}
\renewcommand{\an}{{\mathrm{An}^\dagger}}

\begin{document}

\begin{abstract}
In this paper we define a rigid rational homotopy type, associated to any variety $X$ over a perfect field $k$ of positive characteristic. We prove comparison theorems with previous definitions in the smooth and proper, and log-smooth and proper case. Using these, we can show that if $k$ is a finite field, then the Frobenius structure on the higher rational homotopy groups is mixed. We also define a relative rigid rational homotopy type, and use it to define a homotopy obstruction for the existence of sections.
\end{abstract}

\maketitle

\tableofcontents 

\section{Introduction}
\label{intro}

The object of this paper is the study of rational homotopy types in the context of rigid cohomology. In the first few sections we extend Olsson's and Kim/Hain's definitions of $p$-adic rational homotopy types (see \cite{KH04,Ols07b}) to define the rigid rational homotopy type of an arbitrary $k$-variety $X$, where $k$ is a perfect field of characteristic $p>0$. We do this in two different ways: first using embedding systems and overconvergent de Rham dga's, which is nothing more than an extension of Olsson's methods from the convergent to the overconvergent case, and secondly using Le Stum's overconvergent site. The main focus is on comparison results, comparisons with Olsson's and Kim/Hain's definitions are made, as well as comparisons between the two approaches. We also study Frobenius structures, and use these comparison theorems as well as Kim/Hain's result in the case of a good compactification to prove that the rigid rational homotopy type of a variety over a finite field is mixed. As a corollary of this, we deduce that the higher rational homotopy groups of such varieties are mixed. We also use methods similar to Navarro-Aznar's in the Hodge theoretic context (see \cite{NA87}) to discuss the uniqueness of the weight filtration for Frobenius on rational homotopy types. 

We then turn to the relative rigid rational homotopy type, and again we give two definitions, one in terms of Le Stum's overconvergent site, and the other in terms of framing systems and relative overconvergent de Rham complexes. The comparison between the two should then induce a Gauss--Manin connection on the latter, however, we have so far been unable to prove the required property of the former object, namely that it is `crystalline' in the sense of derived categories. What we can show is that this would follow from a certain `generic coherence' result for Le Stum's relative overconvergent cohomology, of which there are analogues in other versions of $p$-adic cohomology such as the theory of arithmetic $\mathcal{D}$-modules or relative rigid cohomology. Here, our approach is strongly influenced again by Navarro-Aznar in his paper \cite{NA93} on relative de Rham rational homotopy theory.
\tableofcontents

\addtocontents{toc}{\setcounter{tocdepth}{1}}

\section{Differential graded algebras and affine stacks}\label{tools}

In this section we quickly recall some of the tools used by Olsson in \cite{Ols07b} to define homotopy types of varieties in positive characteristic, that is To\"{e}n's theory of affine stacks. Although later we will mainly be focusing on the theory of differential graded algebras, we include this material to emphasize the fact that what we are doing is an extension of a particular case of Olsson's work. We will also need it to prove a comparison theorem between different constructions of unipotent fundamental groups.

Let $K$ be a field of characteristic 0. We will denote by $\mathrm{dga}_K$ the category of unital, graded commutative, differential graded algebras over $K$, concentrated in non-negative degrees. We will denote by $\Delta$ the simplicial category, that is, the category whose objects are ordered sets $[n]=\{0,\ldots,n\}$ and morphisms order preserving maps, and by $\mathrm{Alg}_K^\Delta$ the category of cosimplicial $K$-algebras, that is the category of functors $\Delta\rightarrow \mathrm{Alg}_K$. Let $\mathrm{Aff}_K$ denote the category of affine schemes over $K$, that is the opposite category of $\mathrm{Alg}_K$, which we will endow with the fpqc topology unless otherwise mentioned. We will denote by $\mathrm{Pr}(K)$ (respectively, $\mathrm{Sh}(K)$) the category of presheaves (respectively, sheaves) on $\mathrm{Aff}_K$, and $\mathrm{SPr}(K)$ the category of simplicial presheaves on $\mathrm{Aff}_K$, that is the category of functors $\mathrm{Aff}_K\rightarrow \mathrm{SSet}$ into simplicial sets. There are functors
\begin{align}
D:\mathrm{dga}_K &\rightarrow \mathrm{Alg}_K^\Delta \\
\mathrm{Spec}:(\mathrm{Alg}_K^\Delta)^\circ &\rightarrow \mathrm{SPr}(K) 
\end{align}
where $D$ is the Dold-Kan de-normalization functor (see Chapter 8.4 of \cite{Wei94}).

Suppose that $F\in\mathrm{SPr}(K)$, and $x\in F_0(R)$ for some $R\in\mathrm{Aff}_K$. Then, for all $n\geq 1$, there is a presheaf of groups $\pi_n^\mathrm{pr}(F,x):\mathrm{Aff}_K/R\rightarrow (\mathrm{Groups})$ which takes $S\rightarrow R$ to $\pi_n(\norm{F(S)},x)$ (here, $\norm{\cdot}$ is the geometric realization functor). We define $\pi_n(F,x)$ to be the sheafification of this presheaf. We also define $\pi_0(F)$ to be the sheafification of the presheaf $R\mapsto \pi_0(\norm{F(R)})$.

\begin{definition} A morphism $A^*\rightarrow B^*$ in $\mathrm{dga}_K$ is said to be a:

\begin{itemize} \item weak equivalence if it induces isomorphisms on cohomology;
\item fibration if it is surjective in each degree;
\item cofibration if it satisfies the left lifting property with respect to trivial fibrations.
\end{itemize}

\end{definition}

\begin{definition} A morphism $A^\bullet\rightarrow B^\bullet$ in $\mathrm{Alg}^\Delta_K$ is said to be a:

\begin{itemize} \item weak equivalence if the induced map $H^*(N(A^\bullet))\rightarrow H^*(N(B^\bullet))$ on the cohomology of the normalized complex of the underlying cosimplicial $K$-module is an isomorphism;
\item fibration if it is level-wise surjective;
\item cofibration if it satisfies the left lifting property with respect to trivial fibrations.
\end{itemize}

\end{definition}

\begin{definition} A morphism $F\rightarrow G$ in $\mathrm{SPr}(K)$ is said to be a:

\begin{itemize} \item weak equivalence if it induces isomorphisms on all homotopy groups;
\item cofibration if for every $R\in\mathrm{Aff}_K$, $F(R)\rightarrow G(R)$ is a cofibration in $\mathrm{SSet}$;
\item fibration if it satisfies the right lifting property with respect to trivial cofibrations.
\end{itemize}

\end{definition}

Then, $D$ is an equivalence of model categories, and $\mathrm{Spec}$ is right Quillen (Proposition 2.2.2 of \cite{Toe04}). Thus, we get functors $D$, $\R\mathrm{Spec}$ on the level of homotopy categories.

We will also need the functor of Thom-Sullivan cochains, this is a functor
\begin{equation}
\mathrm{Th}:\mathrm{dga}_K^\Delta\rightarrow \mathrm{dga}_K
\end{equation} 
which is defined as follows (see \S2.11-2.14 of \cite{Ols11}). Let $R_p=\mathcal{O}(\Delta^p_K)$ denote the $K$-algebra of functions on the `algebraic $p$-simplex' $\Delta^p_K$, that is $R_p=K[t_0,\ldots,t_p]/(\sum_i t_i=1)$, which we make into a simplicial $K$-algebra $R_\bu$ in the obvious way. Let $\Omega^*_{\Delta^\bu_K}$ be its de Rham, this is a simplicial dga over $K$.

Let $\mathscr{M}_\Delta$ denote the category where object are morphisms $[m]\rightarrow [n]$ in $\Delta$, and where a morphism from $[m]\rightarrow [n]$ to $[m']\rightarrow [n']$ is a commutative square
\begin{equation}
\xymatrix{
[m]\ar[r] & [n] \ar[d]\\
[m'] \ar[r] \ar[u] & [n'].
}
\end{equation}
Given any $A^{*,\bullet}\in \mathrm{dga}^\Delta_K$, we obtain a functor
\begin{align} {\Omega^*_{\Delta^\bu_K}} \otimes {A^{*,\bullet}}:{\mathscr{M}_\Delta} &\rightarrow  \mathrm{dga}_K \\
([m]\rightarrow [n]) &\mapsto \Omega^*_{\Delta^m_K}\otimes_K A^{*,n} 
\end{align}
and we define $\mathrm{Th}(A^{*,\bullet})$ to be  $\varprojlim (\Omega^*_{\Delta^\bu_K}\otimes A^{*,\bullet})$.

\begin{proposition}[(\cite{Ols11}, Theorem 2.12)] \label{Thom} Let $C_K^{\geq0}$ denote the category of non-negatively graded chain complexes of $K$-modules. There is a natural transformation of functors
\begin{equation}
\xymatrix@=6pt{ \mathrm{dga}^\Delta_K \ar[rrr]^{\mathrm{Th}}\ar[ddd]_{\mathrm{forget}} & & &   \mathrm{dga}_K\ar[ddd]^{\mathrm{forget}} \\ & & & &\\ &\ar@{=>}[ur] & & & \\
(C^{\geq0}_K)^\Delta\ar[rrr]_{\mathrm{Tot}_N}  & & & C^{\geq0}_K }
\end{equation}
where $\mathrm{Tot}_N$ is the functor which takes a cosimplicial chain complex $C^{*,\bullet}$ to the total complex of the normalized double complex $N(C^{*,\bullet})$. Moreover, this natural transformation is a quasi-isomorphism when evaluated on objects of $\mathrm{dga}_K^\Delta$.
\end{proposition}

We will also need to consider derived push-forwards for sheaves of dga's. If $(\mathcal{T},\mathcal{O})$ is a ringed topos, with $\mathcal{O}$ a $\Q$-algebra, then the category $\mathrm{dga}(\mathcal{T};\mathcal{O})$ of $\mathcal{O}$-dga's is a model category, with weak equivalences/fibrations defined to be those morphisms which are weak equivalences/fibrations of the underlying complexes, and cofibrations defined using a lifting property. If $f:(\mathcal{T},\mathcal{O})\rightarrow (\mathcal{T}',\mathcal{O}')$ is a morphism of ringed topoi, and both $\mathcal{O},\mathcal{O}'$ are $\Q$-algebras, with $f^{-1}\mathcal{O}'\rightarrow \mathcal{O}$ flat, then $f_*$ is right Quillen, and hence we can consider the functor $\R f_*$ between homotopy categories of dga's. By the definition of the model category structure on $\mathrm{dga}(\mathcal{T};\mathcal{O})$, taking $\R f_*$ commutes with passing to the underling complex. When there is no likelihood of confusion, we will often write $\mathrm{dga}(\mathcal{O})$ instead of $\mathrm{dga}(\mathcal{T};\mathcal{O})$. We will also write $\mathrm{dga}_R$ when $\mathcal{T}$ is the punctual topos and $R$ is a $\Q$-algebra.

If $k$ is a perfect field of positive characteristic, then we will construct homotopy types by considering dga's on cosimplicial lifts to characteristic zero; thus, we will want to consider cosimplicial `spaces' $V_\bu$ over a field $K$ of characteristic $0$. In this situation, we naturally get a derived functor
\begin{equation}
\R\Gamma:\ho{\mathrm{dga}(V_\bu;K)}\rightarrow \ho{\mathrm{dga}_K^\Delta}
\end{equation}
where the RHS is the category $\mathrm{dga}_K^\Delta$ with level-wise quasi-isomorphisms inverted, and can compose with $\mathrm{Th}(-)$ (which naturally descends to $\ho{\mathrm{dga}_K^\Delta})$ to give
\begin{equation}
\R\Gamma_\mathrm{Th}:=\mathrm{Th}\circ\R\Gamma:\ho{\mathrm{dga}(V_\bu;K)}\rightarrow \ho{\mathrm{dga}_K}.
\end{equation}

\section{Rational homotopy types of varieties}

Let $k$ be a perfect field of characteristic $p>0$, and $K$ a complete, discretely valued field with residue field $k$. We will denote by $\mathcal{V}$ the ring of integers of $K$, and by $\varpi$ a uniformizer. In this section, we will define, for any variety $X/k$ (variety = separated scheme of finite type), a stack $(X/K)_\mathrm{rig}\in\mathrm{Ho}(\mathrm{SPr}(K))$ which represents the rational homotopy type of $X/k$. We essentially use Olsson's methods from \cite{Ols07b}, but replacing `embedding systems' by `framing systems'. This allows us to extend the definition of crystalline (unipotent) schematic homotopy types to non-smooth and non-proper $k$-varieties.

\subsection{The definition of rigid homotopy types}

Throughout, formal $\mathcal{V}$-schemes will be assumed to be $\varpi$-adic, topologically of finite type over $\mathcal{V}$, and separated. A frame over $\mathcal{V}$, as defined by Berthelot, consists of a triple $(U, \overline{U}, \mathscr{U})$ where $U\subset \overline{U}$ is an open embedding of $k$-varieties, and $\overline{U}\subset\mathscr{U}$ is a closed immersion of formal $\mathcal{V}$-schemes (considering $\overline{U}$ as a formal $\mathcal{V}$-scheme via its $k$-variety structure). We say that a frame is smooth if the structure morphism $\mathscr{U}\rightarrow\mathrm{Spf}(\mathcal{V})$ is smooth in some neighbourhood of $U$, and proper if $\mathscr{U}$ is proper over $\mathcal{V}$. We denote the generic fibre of $\mathscr{U}$ in the sense of rigid analytic spaces by $\mathscr{U}_{K0}$; the reason for this being that later on we will want to consider Berkovich spaces, and we need a way to distinguish the two. Let $X/k$ be a variety over $k$.

\begin{definition} A framing system for $X/K$ consists of a simplicial frame $\mathfrak{U}_\bullet=(U_\bullet, \overline{U}_\bullet,\mathscr{U}_\bullet)$ such that:

\begin{itemize} \item $U_\bullet\rightarrow X$ is a Zariski hyper-covering (or an \'{e}tale or proper hyper-covering);
\item for each $n$, $(U_n, \overline{U}_n, \mathscr{U}_n)$ is a smooth and proper frame.
\end{itemize}

\end{definition}

\begin{proposition} \label{frame1} Every pair $X/K$ as above admits a framing system.
\end{proposition}

\begin{proof} Let $\{U_i\}$ be a finite open affine covering for $X$. Then there exists an embedding $U_i\rightarrow \mathbb{P}^{n_i}_k$ for some $n_i$, and we let $\overline{U}_i$ be the closure of $U_i$ in $\mathbb{P}^{n_i}_k$. We can now consider the frame $(U, \overline{U}, \mathscr{U})$ where $U=\CMcoprod_i U_i$, $\overline{U}=\CMcoprod_i\overline{U}_i$ and $\mathscr{U}=\CMcoprod_i \widehat{\mathbb{P}}^{n_i}_{\mathcal{V}}$. Now define $U_n=U\times_X\ldots\times_X U$, with $n$ copies of $U$, and similarly define  $Y_n=\overline{U}\times_k\ldots \times_k \overline{U}$ and $\mathscr{U}_n=\mathscr{U}\times_\mathcal{V}\ldots \times_\mathcal{V}\mathscr{U}$, fibre product in the category of formal $\mathcal{V}$-schemes. Then we have a simplicial triple $(U_\bullet, Y_\bullet,\mathscr{U}_\bullet)$, and we get a framing system $(U_\bullet, \overline{U}_\bullet,\mathscr{P}_\bullet)$ for $X$ by taking $\overline{U}_n$ to be the closure of $U_n$ in $Y_n$. \end{proof}

Given a framing system $\mathfrak{U}_\bullet$ for $X/K$, we get a simplicial  rigid analytic space $V_0(\mathfrak{U}_\bullet):=]\overline{U}_\bullet[_{\mathscr{U}_\bullet0}$ over $K$ (here the $0$ refers to the fact that we are working with rigid, rather than Berkovich spaces), as well as a sheaf of $K$-dga's $j^\dagger\Omega^*_{]\overline{U}_\bullet[_{\mathscr{U}_\bullet0}}$ on this simplicial space. Here, $j^\dagger$ is Berthelot's functor of overconvergent sections.

\begin{definition} The rational homotopy type of $X/K$ is by definition 
\begin{equation}
\R\Gamma_\mathrm{Th}(\Omega^*(\mathcal{O}^\dagger_{X/K})):=\mathrm{Th}(\R\Gamma(j^\dagger\Omega^*_{]\overline{U}_\bullet[_{\mathscr{U}_\bullet0}}))\in\mathrm{Ho}(\mathrm{dga}_K
\end{equation}
and we will denote by $(X/K)_\mathrm{rig}$ the affine stack $\R\mathrm{Spec}(D(\R\Gamma_\mathrm{Th}(\Omega^*(\mathcal{O}^\dagger_{X/K}))))$. We may sometimes refer to $(X/K)_\mathrm{rig}$ as the rational homotopy type of $X$, and will try to keep any confusion this might cause to a minimum.
\end{definition}

\begin{remark} As a rational homotopy type, this definition only captures unipotent information about the fundamental group. In \cite{Ols07b}, Olsson defines a pointed homotopy type that captures the whole pro-algebraic theory of a geometrically connected, smooth and proper $k$-variety. It would not be hard to mimic his methods to give a general definition for an arbitrary geometrically connected $k$-variety, but to do so would involve a choice of base-point. For the most part we want to avoid doing this, which is why we restrict ourselves to rational homotopy types.
\end{remark}

Of course, we must prove that the definition is independent of the framing system $\mathfrak{U}_\bullet$ chosen. The first step is to show that we can recover the rigid cohomology of $X/K$ from $\R\Gamma_{\mathrm{Th}}(\Omega^*(\mathcal{O}^\dagger_{X/K}))$.

\begin{lemma} Consider the forgetful functor $\varphi:\ho{\mathrm{dga}_K}\rightarrow \ho{C^{\geq0}_K}$. Then \begin{equation}
\varphi(\R_{\mathrm{Th}}\Gamma(\Omega^*(\mathcal{O}_{X/K}^\dagger)))\cong \R\Gamma_\mathrm{rig}(X/K)
\end{equation} the latter being the rigid cohomology of $X/K$.
\end{lemma}

\begin{remark} Note that $\ho{C^{\geq0}_K}$ is naturally a full subcategory of the derived category $D^-(K)$.
\end{remark}

\begin{proof}
Using Proposition \ref{Thom}, this just follows from cohomological descent for rigid cohomology, see e.g. Theorem 7.1.2 of \cite{Tsu04}.
\end{proof}

\begin{corollary} The object $\R\Gamma_{\mathrm{Th}}(\Omega^*(\mathcal{O}^\dagger_{X/K}))$ is independent of the framing system chosen. 
\end{corollary}

\begin{proof} The proof is exactly as in \cite{Ols07b}, Section 2.24. If we have two framing systems $\mathfrak{V}_\bullet=(V_\bullet , \overline{V}_\bullet , \mathscr{V}_\bullet)$ and $\mathfrak{U}_\bullet=(U_\bullet, \overline{U}_\bullet, \mathscr{U}_\bullet)$ for $X/K$, then we can take their product  $(U_\bullet\times_X V_\bullet, \overline{U}_\bullet\times_k\overline{V}_\bullet,\mathscr{U}_\bullet\times_\mathcal{V}\mathscr{V}_\bullet)$, and after replacing $\overline{U}_\bullet\times_k \overline{V}_\bullet$ by the closure of $U_\bullet\times_X V_\bullet$, we get a smooth and proper framing system which maps to both $\mathfrak{U}_\bullet$ and $\mathfrak{V}_\bullet$. Hence, we may assume that we have a map $\mathfrak{V}_\bullet\rightarrow \mathfrak{U}_\bullet$. This induces a map $\mathrm{Th}(\R\Gamma(j^\dagger\Omega^*_{V_0(\mathfrak{U}_\bullet)}))\rightarrow \mathrm{Th}(\R\Gamma(j^\dagger\Omega^*_{V_0(\mathfrak{V}_\bullet)}))$ in $\ho{\mathrm{dga}_K}$ and to check that it is an isomorphism, we may forget the algebra structure and prove that it is an isomorphism in $\ho{C_K^{\geq0}}\subset D^-(K)$. But this is true because (after forgetting the algebra structure) both sides compute the rigid cohomology of $X/K$.
\end{proof}

\subsection{Comparison with Navarro-Aznar's construction of homotopy types}

Suppose that our variety $X/k$ is `suitably nice', in that it admits an embedding into a smooth and proper frame $\mathfrak{X}=(X, \overline{X} ,\mathscr{X})$. Then, the work of Navarro-Aznar in \cite{NA87} suggests a closely related, but \emph{a priori} different way of computing the homotopy type of $X/k$. One considers the sheaf of dga's $j^\dagger\Omega^*_{]\overline{X}[_{\mathscr{X}0}}$ on $]\overline{X}[_{\mathscr{X}0}$, and then simply defines the rational homotopy type of $X/k$ to be $\R\Gamma(j^\dagger\Omega^*_{]\overline{X}[_{\mathscr{X}0}})$. That this agrees with the above definition follows from the fact that if $A^\bu\in\ho{\mathrm{dga}_K}^\Delta$ is the constant cosimplicial object on $A$, then $\mathrm{Th}(A^\bu)\cong A$.

\subsection{Comparison with Olsson's homotopy types}\label{compolsson}

Now suppose that $X$ is geometrically connected, smooth and proper, and that $K=\mathrm{Frac}(W(k))$ is the fraction field of the Witt vectors of $k$. Then, Olsson has define a pointed stack $X_\mathscr{C}\in\ho{\mathrm{SPr}_*(K)}$ associated to the category $\mathscr{C}$ of unipotent convergent isocrystals on $X$. In this section, we would like to compare $(X/K)_\mathrm{rig}$ with $X_\mathscr{C}$.

We must therefore review Olsson's construction of $X_\mathscr{C}$. He considers an embedding system for $X$, that is an \'{e}tale hyper-covering $U_\bullet$ of $X$, together with an embedding of $U_\bullet$ into a simplicial $p$-adic formal scheme $\mathscr{P}_\bullet$, which is formally smooth over $W=W(k)$. He then considers the $p$-adic completion $D_\bullet$ of the divided power envelope of $U_\bullet$ in $\mathscr{P}_\bullet$, and considers the sheaf of $K$-dga's $\Omega^*_{D_\bullet}\otimes_W K$ on $\mathscr{P}_\bullet$. He then defines $X_\mathscr{C}$ as the stack $\R\mathrm{Spec}(D(\mathrm{Th}(\R\Gamma(\Omega^*_{D_\bullet}\otimes_W K))))$. If $x\in X(k)$, then $x:\mathrm{Spec}(k)\rightarrow X$ induces a morphism $\R\Gamma(\Omega^*_{D_\bullet}\otimes_W K)\rightarrow K$ and hence makes $X_\mathscr{C}$ naturally into a pointed stack. 

Now, we can choose a framing system $\mathfrak{U}_\bullet=(U_\bullet, \overline{U}_\bullet , \mathscr{U}_\bullet)$ for $X$ such that $(U_\bullet,\mathscr{U}_\bullet)$ is an embedding system for $X$, for example any framing system constructed as in Proposition \ref{frame1} will do. If we let $D_\bullet$ be the $p$-adic completion of the divided power envelope of $U_\bullet$ in $\mathscr{U}_\bullet$, then the canonical map $(D_\bullet)_{K0}\rightarrow \mathscr{U}_{K0}$ factors through $]U_\bullet[_{\mathscr{U}_\bullet0}$ and hence we get a natural morphism
\begin{equation}
\R\Gamma (j^\dagger\Omega^*_{]\overline{U}_\bullet[_{\mathscr{U}_\bullet0}} )\rightarrow \R\Gamma(\Omega^*_{D_\bullet}\otimes_W K)
\end{equation}
in $\ho{\mathrm{dga}_K^\Delta}$. We claim that it becomes an isomorphism after applying $\mathrm{Th}(-)$. Indeed, we may forget the algebra structure and prove that it is an isomorphism in $\ho{\mathrm{Ch}^{\geq0}_K}$. But the the LHS computes the rigid cohomology of $X/K$, and the RHS the convergent cohomology of $X/K$. Since $X$ is proper, they coincide.

\subsection{Functoriality and Frobenius structures}

In this section, we discuss the functoriality of the rational homotopy type, as well as how to put a Frobenius structure on the rational homotopy type of a $k$-variety $X$.

So suppose that $f:X\rightarrow Y$ is a morphism of $k$-varieties, $\mathfrak{U}_\bu=(U_\bu,\overline{U}_\bu,\mathscr{U}_\bu)$ is a framing system for $X$, $\mathfrak{V}_\bu=(V_\bu,\overline{V}_\bu,\mathscr{V}_\bu)$ is a framing system for $Y$, and $\mathfrak{f}:\mathfrak{U}_\bu\rightarrow \mathfrak{V}_\bu$ is a morphism covering $f:X\rightarrow Y$. Note that given $f:X\rightarrow Y$ we can always choose such a set-up. Then, we get a morphism
\begin{equation}
 \mathfrak{f}^*_{K0}:j^\dagger\Omega^*_{V_0(\mathfrak{V}_\bu)}\rightarrow j^\dagger\Omega^*_{V_0(\mathfrak{U}_\bu)}
\end{equation}
in $\mathrm{dga}(V_0(\mathfrak{U}_\bu);K)$ which induces a morphism 
\begin{equation}
\mathfrak{f}^*:\R\Gamma_\mathrm{Th}(j^\dagger\Omega^*_{V_0(\mathfrak{V}_\bu)})\rightarrow \R\Gamma_\mathrm{Th}(j^\dagger\Omega^*_{V_0(\mathfrak{U}_\bu)})
\end{equation} 
in $\ho{\mathrm{dga}_K}$. Of course, we need to check that this is independent of the choice of $\mathfrak{f}$, we will not do this here but wait until \S3 when we will have an alternative construction of the rational homotopy type which is clearly functorial. We will, however, still speak of the induced morphism
\begin{equation}
f^*:\R\Gamma_\mathrm{Th}(\Omega^*(\mathcal{O}_{Y/K}^\dagger))\rightarrow \R\Gamma_\mathrm{Th}(\Omega^*(\mathcal{O}_{X/K}^\dagger))
\end{equation} 
in $\ho{\mathrm{dga}_K}$. We can also use similar ideas to define Frobenius structures.

\begin{definition} An $F$-framing of $X$ is a framing $\mathfrak{U}_\bu=(U_\bu,\overline{U}_\bu,\mathscr{U}_\bu)$ as above, together with a lifting $F_\bu:\mathscr{U}_\bu\rightarrow \mathscr{U}_\bu$ of Frobenius compatible with the Frobenius on $K$.
\end{definition}

Given such an $F_\bu$, we get a quasi-isomorphism $F_\bu^*:j^\dagger\Omega^*_{V_0(\mathfrak{U}_\bu)}\otimes_{K,\sigma}K\rightarrow j^\dagger\Omega^*_{V_0(\mathfrak{U}_\bu)}$ in $\mathrm{dga}(V_0(\mathfrak{U}_\bu);K)$ and hence a isomorphism 
\begin{equation}
\phi:\R\Gamma_\mathrm{Th}(\Omega^*(\mathcal{O}_{X/K}^\dagger))\otimes_{K,\sigma}K\rightarrow\R\Gamma_\mathrm{Th}(\Omega^*(\mathcal{O}_{X/K}^\dagger))
\end{equation} 
in $\ho{\mathrm{dga}_K}$. Again, this seemingly depended on the choice of Frobenius $F_\bu:\mathscr{U}_\bu\rightarrow \mathscr{U}_\bu$, and we will prove in \S3 that it does not. Moreover, if $X\rightarrow Y$ is a morphism of $k$-varieties, then we will see that the induced morphism
\begin{equation}
 \R\Gamma_\mathrm{Th}(\Omega^*(\mathcal{O}_{Y/K}^\dagger))\rightarrow \R\Gamma_\mathrm{Th}(\Omega^*(\mathcal{O}_{X/K}^\dagger))
\end{equation}
is compatible with Frobenius, in the sense that we get a commutative diagram
\begin{equation}
\xymatrix{  \R\Gamma_\mathrm{Th}(\Omega^*(\mathcal{O}_{Y/K}^\dagger))\otimes_{K,\sigma}K \ar[r]\ar[d] & \R\Gamma_\mathrm{Th}(\Omega^*(\mathcal{O}_{X/K}^\dagger))\otimes_{K,\sigma}K \ar[d] \\ \R\Gamma_\mathrm{Th}(\Omega^*(\mathcal{O}_{Y/K}^\dagger))\ar[r] & \R\Gamma_\mathrm{Th}(\Omega^*(\mathcal{O}_{X/K}^\dagger)) }
\end{equation} 
in $\ho{\mathrm{dga}_K}$. 

\subsection{Mixedness for homotopy types} In this section, we will suppose that $k=\F_q$ is a finite field, and that $K$ is the fraction field of the Witt vectors $W=W(k)$ of $k$. By Frobenius, we will mean the $q$-power Frobenius. In \S6 of \cite{KH04}, Kim and Hain define mixedness for an $F$-dga, and prove that if $X/k$ is a geometrically connected, smooth $k$-variety, with good compactification, then the $F$-dga that they define to represent the rational homotopy type of $X$ is mixed. We wish to extend their results to show that the rigid rational homotopy type of any $k$-variety $X$ is mixed, and the proof is in three steps.

\begin{itemize} 
\item A comparison between our rigid homotopy type and their crystalline homotopy type, when both are defined. 
\item A descent result for rigid homotopy types, which will follow easily from the corresponding theorem in cohomology. 
\item A result stating that mixedness is preserved under this descent operation.
\end{itemize}

So let $(Y,M)$ be a geometrically connected, log-smooth and proper $k$-variety, such that the log structure $M$ comes from a strict normal crossings divisor $D\subset Y$. We refer the reader to \emph{loc. cit.} for the definition of the crystalline rational homotopy type $A_{(Y,M)}$ of $(Y,M)$ - this is a $K$-dga with a Frobenius structure. Let $Y^\circ=Y\setminus D$ be the complement of $D$.

\begin{proposition} \label{complog}There is a quasi-isomorphism $\R\Gamma_\mathrm{Th}(\Omega^*(\mathcal{O}^\dagger_{Y^\circ/K}))\cong A_{(Y,M)}$ in $F\text{-}\ho{\mathrm{dga}_K}$.
\end{proposition}

\begin{proof} Choose a finite open affine covering $\{ U_i\}$ of $Y$, let $Y_0=\CMcoprod_i U_i$ and let $Y_\bu\rightarrow Y$ be the associated \v{C}ech hyper-covering. The pullback to $Y_\bu$ of the log structure on $Y$ is defined by some strict normal crossings divisor $H_\bu\subset Y_\bu$. Since everything is affine, both $Y_\bu$ and the divisor defining the log structure lift to characteristic zero, so we can choose an exact closed immersion  $Y_\bu\rightarrow Z_\bu$ into a smooth simplicial log scheme over $W$. Let $D_\bu$ be the divided power envelope of $Y_\bu$ in $Z_\bu$, and $\widehat{D}_\bu$ its $p$-adic completion. We have a diagram of cosimplicial dga's
\begin{equation}
\xymatrix{ \R\Gamma(]Y_\bu[_{\widehat{Z}_\bu0},j^\dagger\Omega^*_{]Y_\bu[_{\widehat{Z}_\bu0}}) & \ar[l]\R\Gamma(]Y_\bu[^{\log}_{\widehat{Z}_\bu0},\omega^{*}_{]Y_\bu[^{\log}_{\widehat{Z}_\bu0}}) \ar[d]\\ \R\Gamma(D_{\bu K},\omega^{*}_{D_\bu} \otimes_W K )\ar[r] & \R\Gamma(\widehat{D}_{\bu K},\omega^{*}_{\widehat{D}_\bu} \otimes_W K ) }
\end{equation} 
where
\begin{itemize} 
\item the rigid space $]Y_\bu[^{\log}_{\widehat{Z}_\bu0}$ together with its logarithmic de Rham complex is defined as in \S2.2 of \cite{Shi02};
\item the top horizontal arrow comes from the natural morphism $]Y_\bu[^{\log}_{\widehat{Z}_\bu0}\rightarrow ]Y_\bu[_{\widehat{Z}_\bu0}$;
\item the right-hand side vertical arrow comes from the fact that writing $\omega^{*}_{\widehat{Z}_\bu}$ for the logarithmic de Rham complex on $\widehat{Z}_\bu$,
\begin{align}
\omega^{*}_{]Y_\bu[^{\log}_{\widehat{Z}_\bu0}} &\cong (\omega^{*}_{\widehat{Z}_\bu}\otimes_W K)|_{]Y_\bu[^{\log}_{\widehat{Z}_\bu0}} \\
\omega^{*}_{\widehat{D}_\bu}&\cong \omega^*_{\widehat{Z}_\bu}\otimes_{\mathcal{O}_{\widehat{Z}_\bu}}\mathcal{O}_{\widehat{D}_\bu}  
\end{align}
and the natural map $\widehat{D}_{\bu K}\rightarrow \widehat{Z}_{\bu K}$ factors though $]Y_\bu[^{\log}_{\widehat{Z}_\bu0}$;
\item the bottom horizontal arrow is given by $p$-adic completion.
\end{itemize}

We now apply the functor $\mathrm{Th}(-)$ to obtain the diagram 
\begin{equation}
\xymatrix{ \R\Gamma_\mathrm{Th}(\Omega^*(\mathcal{O}^\dagger_{Y^\circ /K})) & \ar[l]\mathrm{Th}\left(\R\Gamma(]Y_\bu[^{\log}_{\widehat{Z}_\bu0},\omega^{*}_{]Y_\bu[^{\log}_{\widehat{Z}_\bu0}})\right) \ar[d]\\ A_{(Y,M)} \ar[r] & \mathrm{Th}\left(\R\Gamma(\widehat{D}_{\bu K},\omega^{*}_{\widehat{D}_\bu} \otimes_W K )\right) }
\end{equation} 
where the isomorphism
\begin{equation}
\mathrm{Th}\left(\R\Gamma(]Y_\bu[_{\widehat{Z}_\bu0},j^\dagger\Omega^*_{]Y_\bu[_{\widehat{Z}_\bu0}})\right)\cong  \R\Gamma_\mathrm{Th}(\Omega^*(\mathcal{O}^\dagger_{Y^\circ/K}))
\end{equation}
comes from using cohomological descent for partially overconvergent cohomology and the isomorphism
\begin{equation}
\mathrm{Th}\left(\R\Gamma(D_{\bu K},\omega^{*}_{D_\bu} \otimes_W K )\right)\cong A_{(Y,M)} 
\end{equation}
is in \S4 of \cite{KH04}.

I claim that all these morphisms are in fact quasi-isomorphisms. Indeed, the cohomology groups of the top left dga are rigid cohomology groups of $Y^\circ$, those of the top right are the log-analytic cohomology groups of $(Y,M)$ in the sense of Chapter 2 of \cite{Shi02}, those of the bottom right are log-convergent cohomology groups of $(Y,M)$, and those of the bottom left are log-crystalline cohomology groups of $(Y,M)$, tensored with $K$. 

On cohomology, the top horizontal and right vertical arrows are the comparison maps between rigid and log-analytic cohomology and log-analytic and log-convergent cohomology defined in \S\S2.4 and 2.3 of \emph{loc. cit.}, respectively, where they are proved to be isomorphisms. The bottom horizontal arrow is the comparison map between log-crystalline and log-convergent cohomology, which is proved to be an isomorphism in \emph{loc. cit}. \end{proof}

Now let $X$ be a $k$-variety, and $Y_\bu\rightarrow X$ a simplicial $k$-variety mapping to $X$. Then we get an augmented cosimplicial object \begin{equation}
\R\Gamma_{\mathrm{Th}}(\Omega^*(\mathcal{O}_{X/K}^\dagger))\rightarrow \R\Gamma_{\mathrm{Th}}(\Omega^*(\mathcal{O}_{Y_\bu/K}^\dagger))
\end{equation}
in $F\text{-}\ho{\mathrm{dga}_K}$, which induces a morphism
\begin{equation}
\R\Gamma_{\mathrm{Th}}(\Omega^*(\mathcal{O}_{X/K}^\dagger))\rightarrow \mathrm{Th}(\R\Gamma_{\mathrm{Th}}(\Omega^*(\mathcal{O}_{Y_\bu/K}^\dagger))).
\end{equation}
The descent theorem we will need is the following proposition.

\begin{proposition} \label{descenthom}Suppose that $Y_\bu\rightarrow X$ is a proper hyper-covering. Then 
\begin{equation}
\R\Gamma_{\mathrm{Th}}(\Omega^*(\mathcal{O}_{X/K}^\dagger))\rightarrow \mathrm{Th}(\R\Gamma_{\mathrm{Th}}(\Omega^*(\mathcal{O}_{Y_\bu/K}^\dagger)))
\end{equation}
is an isomorphism in $F\text{-}\ho{\mathrm{dga}_K}$.
\end{proposition}

\begin{proof} We may obviously ignore both the $F$-structure, and the algebra structure. But now it follows from Proposition \ref{Thom} together with cohomological descent for rigid cohomology that the induced morphism on cohomology is an isomorphism.
\end{proof}

\begin{remark} The reader might object that $\mathrm{Th}(-)$ does not make sense as a functor on $\ho{\mathrm{dga}_K}^\Delta$. However, this does not matter for us since in the only place where we wish to apply this result (namely Theorem \ref{mixedtheo}) below, we have a specific object of $\mathrm{dga}_K^\Delta$ representing $\R\Gamma_{\mathrm{Th}}(\Omega^*(\mathcal{O}_{Y_\bu/K}^\dagger))\in\ho{\mathrm{dga}_K}^\Delta$.
\end{remark}

We now recall Kim and Hain's definition of mixedness for an $F$-dga over $K$.

\begin{definition} We say that $A\in F\text{-}\mathrm{dga}_K$ is mixed if there exists a quasi-isomorphism $A\simeq B$ in $F\text{-}\mathrm{dga}_K$ and a multiplicative filtration $W^\bu B$ of $B$ such that $H^{p-q}(\mathrm{Gr}^W_p(B))$ is pure of weight $q$ for all $p,q$. We say $A$ is strongly mixed if we can choose the filtration on $A$ itself.
\end{definition}

\begin{lemma} \label{thommixed}Let $A^\bu$ be a cosimplicial $K$-dga with Frobenius action, such that each $A^n$ is strongly mixed. Assume moreover that the cosimplicial structure is compatible with the filtrations. Then $\mathrm{Th}(A^\bu)$ is mixed. 
\end{lemma}

\begin{proof} Let us first forget the algebra structure on $A^\bu$, and treat it as just a cosimplicial complex of $K$-modules. We then have two filtrations on $A^\bu$ - one coming from the weight filtration $W$ on each $A^n$, and the other coming from the filtration by simplicial degree. This induces two filtrations $W$ and $D$ on $\mathrm{Tot}_N(A^\bu):=\mathrm{Tot}(N(A^\bu))$ and we define $F$ to be the convolution $D*W$ of these filtrations. We can similarly define the filtration $F$ on the un-normalized total complex $\mathrm{Tot}(A^\bu)$ (where the chain maps in one direction are the alternating sums of the coface maps), and there is a filtered quasi-isomorphism 
\begin{equation}
\mathrm{Tot}(A^\bu)\simeq\mathrm{Tot}_N(A^\bu)
\end{equation}
arising from the usual comparison of $\mathrm{Tot}$ and $\mathrm{Tot}_N$. We can now calculate
\begin{align} H^{p-q}(\mathrm{Gr}^F_p \mathrm{Tot}_N(A^\bu)) &= H^{p-q}(\mathrm{Gr}^F_p \mathrm{Tot}(A^\bu)) \\ &= H^{p-q}( \mathrm{Tot}(\mathrm{Gr}^F_pA^\bu)) \\&= \bigoplus_{i+j=p} H^{p-q} (\mathrm{Tot}(\mathrm{Gr}^D_i\mathrm{Gr}^W_j A^\bu )) \\ &= \bigoplus_{i} H^{p-i-q} (\mathrm{Gr}^W_{p-i}A^i ) 
\end{align} which is pure of weight $q$. Now, to take account of the multiplicative structure on $A^\bu$, we simply use Lemme 6.4 of \cite{NA87}, which says that the complex $\mathrm{Tot}_N(A^\bu)$ considered above, with the filtration $D*W$, is filtered quasi-isomorphic (as a filtered complex) to $\mathrm{Th}(A^\bu)$ with a certain naturally defined multiplicative filtration.
\end{proof}

The proof that the rigid rational homotopy type is mixed is now straightforward.

\begin{theorem} \label{mixedtheo} Let $k$ be a finite field, and $K=\mathrm{Frac}(W(k))$. Let $X$ be a geometrically connected $k$-variety. Then the rational homotopy type $\R\Gamma_{\mathrm{Th}}(\Omega^*(\mathcal{O}^\dagger_{X/K}))$ is mixed.  
\end{theorem}

\begin{proof} By de Jong's theorem on alterations, there exists a proper hyper-covering $Y_\bu\rightarrow X$ such that $X_\bu$ admits a good compactification, that is an embedding $Y_\bu\rightarrow \overline{Y}_\bu$ into a smooth and proper simplicial $k$-scheme with complement a strict normal crossings divisor on each level $\overline{Y}_n$. Let $M_n$ be the log structure associated to this divisor. By Propositions \ref{complog} and \ref{descenthom}, we have a quasi-isomorphism 
\begin{equation}
\R\Gamma_\mathrm{Th}(\Omega^*(\mathcal{O}^\dagger_{X/K}))\cong \mathrm{Th}\left(A_{(\overline{Y}_\bu,M_\bu)}\right)
\end{equation}
of dga's with Frobenius. Let $\spec{k}^\circ$ denote the scheme $\spec{k}$ with the log structure of the punctured point, and let $(\overline{Y}_\bu,M_\bu^\circ)$ denote the pullback of $(\overline{Y}_\bu,M_\bu)$ via the natural morphism $\spec{k}^\circ\rightarrow \spec{k}$. Since log-crystalline cohomology in \cite{KH04} is calculated relative to the log structure induced on $\spec{W(k)}$ via the Teichm\"{u}ller lift from that on $\spec{k}$, it follows that there is a Frobenius invariant, level-wise quasi-isomorphism
\begin{equation} A_{(\overline{Y}_\bu,M_\bu^\circ)}\cong A_{(\overline{Y}_\bu,M_\bu)}
\end{equation}
as cosimplicial dga's. Hence we also have a quasi-isomorphism
\begin{equation}
\R\Gamma_\mathrm{Th}(\Omega^*(\mathcal{O}^\dagger_{X/K}))\cong \mathrm{Th}\left(A_{(\overline{Y}_\bu,M_\bu^\circ)}\right)
\end{equation} 
of dga's with Frobenius. Now, although each $A_{(\overline{Y}_n,M^\circ_n)}$ is not strongly mixed, each is quasi-isomorphic to one that is, let us call it $\tilde{A}_{(\overline{Y}_n,M^\circ_n)}$ (this is the dga $TW(W\tilde{\omega}[u])$ in the notation of \emph{loc. cit.} - note that since we are assuming that $Y$ is smooth, we can work with the dga $W\tilde{\omega}[u]$ rather than $C(W\tilde{\omega}[u])$). This dga is functorial in $(Y,M)$ in exact the same manner as $A_{(\overline{Y},M^\circ)}$. Moreover, the weight filtrations on these dga's are also functorial, and hence the result now follows from Lemma \ref{thommixed} and the corresponding result in the log-smooth and proper case, which is Theorem 3 of \emph{loc. cit}.
\end{proof}

\begin{remark} In what follows we will generally replace $A_{(\overline{Y},M)}$ by this quasi-isomorphic strongly mixed complex; since the latter is functorial in $(\overline{Y},M)$ this will not cause any problems.
\end{remark}

\begin{remark} Strictly speaking, Kim and Hain's definition cannot be applied to $\R\Gamma_\mathrm{Th}(\Omega^*(\mathcal{O}^\dagger_{X/K}))$ since the Frobenius action is only in the homotopy category. However, Theorem 3.47 of \cite{Ols07b} allows us to lift this action to the category $\mathrm{dga}_K$, uniquely up to quasi-isomorphism. Alternatively, since we have a Frobenius action on each dga $A_{(\overline{Y}_n,M_n)}$, we can use this to put a Frobenius action on $\mathrm{Th}(A_{(\overline{Y}_\bu,M_\bu)})$. Proposition  \ref{descenthom} would then say that after applying the functor $F\text{-}\mathrm{dga}_K\rightarrow F\text{-}\ho{\mathrm{dga}_K}$ this is isomorphic to $\R\Gamma_\mathrm{Th}(\Omega^*(\mathcal{O}^\dagger_{X/K}))$.
\end{remark}

If $x\in X(k)$ is a point, then we can use similar methods to the previous section to define an object $\R\Gamma_\mathrm{Th}(\Omega^*(\mathcal{O}^\dagger_{X/K}),x)$ in the homotopy category of augmented $F$-dga's over $K$, where the augmentation comes from `pulling back' to the point $x$. All the above comparison isomorphisms go through in this augmented situation, as does the definition of mixedness. Thus, as in \S6 of \cite{KH04}, if $X/k$ is geometrically connected, then the bar complex $B(\R\Gamma_\mathrm{Th}(\Omega^*(\mathcal{O}^\dagger_{X/K}),x))$ associated to the augmented $F$-dga $\R\Gamma_\mathrm{Th}(\Omega^*(\mathcal{O}^\dagger_{X/K}),x)$ is mixed.

Recall that we define the homotopy groups of $X/k$ by 
\begin{align}
\pi_1^\mathrm{rig}(X,x)&=\mathrm{Spec}(H^0(B(\R\Gamma_\mathrm{Th}(\Omega^*(\mathcal{O}^\dagger_{X/K}),x)))) \\ \pi_n^\mathrm{rig}(X,x)&= (Q H^{n-1}(B(\R\Gamma_\mathrm{Th}(\Omega^*(\mathcal{O}^\dagger_{X/K}),x))))^\vee,\quad n\geq2.
\end{align}
where $Q$ is the functor of indecomposable cohomology classes.

\begin{corollary} Let $X/k$ be a geometrically connected variety, and $x\in X(k)$. Then the rational homotopy groups $\pi_n^\mathrm{rig}(X,x)$ are mixed for all $n\geq 1$.
\end{corollary}

\begin{remark} For $n=1$ we mean by this that $H^0(B(\R\Gamma_\mathrm{Th}(\Omega^*(\mathcal{O}^\dagger_{X/K}),x)))$ is mixed.
\end{remark}

Although we have proved that there is a mixed structure on the rational homotopy type of a $k$-variety $X$, in order to define such a structure, we chose a log-smooth and proper resolution $(\overline{Y}_\bu,M_\bu)\rightarrow X$ of $X$. Hence \emph{a priori} the filtration that we have on $\R\Gamma_\mathrm{Th}(\Omega^*(\mathcal{O}^\dagger_{X/K}))$ depends on this resolution. Thus the question remains of how `independent' this structure is of the resolution chosen. In order to answer this question, we will need to talk about the different notions of equivalence for filtered dga's, as well as tidying up the slightly sloppy definition of the mixed structure on $\R\Gamma_\mathrm{Th}(\Omega^*(\mathcal{O}^\dagger_{X/K}))$ given above.

\begin{remark} It is fairly simple to show that induced filtration on the rational homotopy groups $\pi_n^\mathrm{rig}(X,x)$ are independent of the chosen resolution, however, we would like a similar result about the whole dga $\R\Gamma_\mathrm{Th}(\Omega^*(\mathcal{O}^\dagger_{X/K}))$. We will then deduce the result about the homotopy groups as a simple corollary.
\end{remark}

Suppose that $f:A\rightarrow B$ is a filtered morphism between filtered dga's. That is $A$ and $B$ are equipped with multiplicative filtrations, and $f$ is compatible with the filtrations. Thus $f$ defines a morphism 
\begin{equation}
E^{\bu,\bu}_1(f):E^{\bu,\bu}_1(A)\rightarrow E^{\bu,\bu}_1(B)
\end{equation} 
between the $E_1$-pages of the spectral sequences associated to the filtrations on $A$ and $B$.

\begin{definition} We say that $f$ is an $E_r$ quasi-isomorphism if $E^{p,q}_{r+1}(f)$ is an isomorphism for all $p,q$.
\end{definition}

\begin{remark} The notion of filtered quasi-isomorphism of filtered complexes used above exactly corresponds to an $E_0$-quasi-isomorphism. It is also worth noting that filtered dga's do not form a model category.
\end{remark}

We want to consider the following categories, as well as the obvious augmented versions.

\begin{itemize} 
\item $F\text{-}\ho{\mathrm{dga}_K}$, the category of $F$-objects in $\ho{\mathrm{dga}_K}$. This is where the rational homotopy type $\R\Gamma_\mathrm{Th}(\Omega^*(\mathcal{O}^\dagger_{X/K}))$ lives;
\item $F^\mathcal{M}\text{-}\mathrm{dga}_K$, - the category of mixed Frobenius dga's over $K$, that is, Frobenius dga's with a filtration such that $H^{q-p}(\mathrm{Gr}_p(-))$ is pure of weight $q$. Owing to the work of Kim and Hain, for $(\overline{Y},D)$  a smooth and proper $k$-variety with strict normal crossings divisor $D$, we can view the rational homotopy type $A_{(\overline{Y},D)}$ functorially as an object in this category;
\item for each $r\geq0$, the category $\mathrm{Ho}_r(F^\mathcal{M}\text{-}\mathrm{dga}_K)$ which is the localization of $F^\mathcal{M}\text{-}\mathrm{dga}_K$ with respect to $E_r$-quasi-isomorphisms;
\end{itemize}

Since an $E_r$-quasi-isomorphism is always a quasi-isomorphism, there are obvious forgetful functors
\begin{equation}
\mathrm{Ho}_r(F^\mathcal{M}\text{-}\mathrm{dga}_K)\rightarrow F\text{-}\ho{\mathrm{dga}_K}
\end{equation}
for each $r$. Choosing a resolution $(\overline{Y}_\bu,D_\bu)\rightarrow X$ of a $k$-variety $X$, we get an isomorphism 
\begin{equation}
\mathrm{Th}(A_{(\overline{Y}_\bu,D_\bu)})\cong\R\Gamma_\mathrm{Th}(\Omega^*(\mathcal{O}^\dagger_{X/K}))
\end{equation}
in $F\text{-}\ho{\mathrm{dga}_K}$. The question then remains, in what sense is $\mathrm{Th}(A_{(\overline{Y}_\bu,D_\bu)})$ independent of the resolution chosen?

\begin{lemma} The object $\mathrm{Th}(A_{(\overline{Y}_\bu,D_\bu)})$ in the localized category $\mathrm{Ho}_1(F^\mathcal{M}\text{-}\mathrm{dga}_K)$ depends only on $X$.
\end{lemma}

\begin{remark}Hence, we may view $ \R\Gamma_\mathrm{Th}(\Omega^*(\mathcal{O}^\dagger_{X/K}))$ canonically as an object of $\mathrm{Ho}_1(F^\mathcal{M}\text{-}\mathrm{dga}_K)$.
\end{remark}

\begin{proof} Since for any two resolutions, we can find a third mapping to both, it suffices to prove that any quasi-isomorphism between mixed complexes is in fact an $E_1$-quasi-isomorphism. But this follows easily from the fact that the spectral sequence degenerates at the $E_2$-page.
\end{proof}

Of course, in the same manner, for any rational point $x\in X(k)$ we can view the augmented dga $ \R\Gamma_\mathrm{Th}(\Omega^*(\mathcal{O}^\dagger_{X/K}),x)$ as an object in the category \begin{equation}
\mathrm{Ho}_1(F^\mathcal{M}\text{-}\mathrm{dga}^*_K)
\end{equation}
where the `$*$' refers to the fact that we are considering augmented dga's. 

Let $\mathrm{DGA}_K$ denote the category of commutative dga's over $K$ that are not necessarily concentrated in non-negative degrees, we will use similar notation for unbounded mixed Frobenius dga's.  As proved in \S6 of \cite{KH04}, the bar construction for dga's can be extended to a functor 
\begin{equation}
B: F^\mathcal{M}\text{-}\mathrm{dga}^*_K\rightarrow F^\mathcal{M}\text{-}\mathrm{DGA}_K
\end{equation}
or in other words, the bar complex of a mixed, augmented Frobenius dga is a mixed Frobenius dga. We can also consider the cohomology functor
\begin{equation}
H^*: F^\mathcal{M}\text{-}\mathrm{dga}_K\rightarrow F^\mathcal{M}\text{-}\mathrm{dga}_K
\end{equation}
as well as the corresponding version for unbounded dga's. We let $\mathrm{Ho}_1^\mathrm{con}(F^\mathcal{M}\text{-}\mathrm{dga}_K^*)$ denote the localized category of dga's with connected cohomology.

\begin{lemma} We have factorizations
\begin{equation}
B:\mathrm{Ho}_1^\mathrm{con}(F^\mathcal{M}\text{-}\mathrm{dga}_K^*)\rightarrow \mathrm{Ho}_1(F^\mathcal{M}\text{-}\mathrm{DGA}_K)
\end{equation}  
\begin{equation}
H^*:\mathrm{Ho}_1(F^\mathcal{M}\text{-}\mathrm{dga}_K)\rightarrow F^\mathcal{M}\text{-}\mathrm{dga}_K
\end{equation}
\begin{equation}
H^*:\mathrm{Ho}_1(F^\mathcal{M}\text{-}\mathrm{DGA}_K)\rightarrow F^\mathcal{M}\text{-}\mathrm{DGA}_K
\end{equation}
\end{lemma}

\begin{proof} By the proof of the previous lemma (any quasi-isomorphism between mixed complexes is an $E_1$-quasi-isomorphism - this applies to unbounded dga's as well), the first factorization follows from the fact that the bar complex sends quasi-isomorphisms between dga's with connected cohomology to quasi-isomorphisms. The seconds and third factorizations are easy, and in fact hold with $\mathrm{Ho}_1$ replaced by $\mathrm{Ho}_r$ for any $r\geq0$.
\end{proof}

\begin{corollary} Let $X/k$ be geometrically connected. Then the mixed structures on the cohomology ring $H^*_\mathrm{rig}(X/K)\in F^\mathcal{M}\text{-}\mathrm{dga}_K$ and the homotopy groups $\pi_n^\mathrm{rig}(X,x)$, $n\geq1$ for any $x\in X(k)$, are independent of the resolution chosen.
\end{corollary}

\subsection{Homotopy obstructions}\label{homobssec}

We now briefly discuss a crystalline homotopy obstruction to the existence of maps between varieties, and of sections of maps between $k$-varieties, which is nothing more than an application of the functoriality of the previous section. For any variety $X/k$, $\R\Gamma_\mathrm{Th}(\Omega^*(\mathcal{O}_{X/K}^\dagger))$ is naturally an object of $F\text{-}\ho{\mathrm{dga}_K}$, and hence for any two varieties $X,Y$ we can consider the set 
\begin{equation}
 [\R\Gamma_\mathrm{Th}(\Omega^*(\mathcal{O}_{X/K}^\dagger)),\R\Gamma_\mathrm{Th}(\Omega^*(\mathcal{O}_{Y/K}^\dagger))]_{F\text{-}\ho{\mathrm{dga}_K}}
\end{equation}
of morphisms $\R\Gamma_\mathrm{Th}(\Omega^*(\mathcal{O}_{X/K}^\dagger))\rightarrow \R\Gamma_\mathrm{Th}(\Omega^*(\mathcal{O}_{Y/K}^\dagger))$ in $F\text{-}\ho{\mathrm{dga}_K}$. Functoriality will induce a map 
\begin{equation}
\mathrm{Mor}_{\mathrm{Sch}/k}(Y,X)\rightarrow [\R\Gamma_\mathrm{Th}(\Omega^*(\mathcal{O}_{X/K}^\dagger)),\R\Gamma_\mathrm{Th}(\Omega^*(\mathcal{O}_{Y/K}^\dagger))]_{F\text{-}\ho{\mathrm{dga}_K}}.
\end{equation} 
and we can use this to study the set of maps from $Y$ to $X$. Of course, if we are given a map $X\rightarrow Y$, then we can use a similar approach to study sections of this map. 

We will not pursue this idea, since we actually wish to develop a more refined homotopical approach to studying sections. To motivate why this better approach is needed, consider the morphism
\begin{equation}
\A^1_k\rightarrow \A^1_k,\quad x\mapsto x^2
\end{equation} 
which clearly does not have a section. However, we cannot detect this on the level of rational homotopy types, since $\R\Gamma_\mathrm{Th}(\Omega^*(\mathcal{O}^\dagger_{\A^1_k/K}))=K$. Instead, we will develop a relative rational homotopy type which will associate to any morphism $X\rightarrow Y$ a dga on $Y$ (in a sense that will be made clear later) in a functorial manner. Before we do so, however, we will first give an alternative perspective on the rigid rational homotopy type.

\section{Overconvergent sheaves and homotopy types}\label{overcon}

In this section we wish to describe a different way to construct the rational homotopy type of a $k$-variety $X$, using the theory of modules on a certain `overconvergent' site attached to $X/K$, as developed byLLe Stum. To motivate this slightly altered perspective, it may be helpful to discuss the analogous situation in characteristic zero. So let $X/\C$ be a smooth, proper algebraic variety, then the rational homotopy type of $X$ is defined to be $\R\Gamma_\mathrm{Zar}(\Omega^*_X)$, using similar methods to those we have seen already. Why does this give the `right' answer? 

The reason is that after passing to the analytic topology of $X$, $\Omega^*_X$ is quasi-isomorphic, as a dga, to the constant sheaf of dga's $\underline{\C}$, and standard theorems comparing Zariski and analytic cohomology of coherent sheaves will then give us an isomorphism $\R\Gamma_\mathrm{an}(\underline{\C})\cong\R\Gamma_\mathrm{Zar}(\Omega^*_X)$ in $\ho{\mathrm{dga}_\C}$. The former is then the `correct' rational homotopy type of $X$, essentially because of Th\'{e}or\`{e}me 5.5 of \cite{NA93}. So we have a functor $\R\Gamma_\mathrm{an}:\ho{\mathrm{dga}(X^\mathrm{an};\C)}\rightarrow \ho{\mathrm{dga}_\C}$, and $\R\Gamma_\mathrm{Zar}(\Omega^*_X)$ gives us a way of \emph{computing} $\R\Gamma_\mathrm{an}(\underline{\C})$ in an algebraic fashion. 

There is now an obvious third candidate for defining the rational homotopy type of $X$ - we consider the constant crystal $\mathcal{O}_{X/\C}$ on the infinitesimal site of $X/\C$, and simply take $\R\Gamma_\mathrm{inf}(\mathcal{O}_{X/\C})$ in the sense of dga's, rather than complexes. We can then trace through Grothendieck's comparison theorems to show that this is naturally isomorphic to $\R\Gamma_\mathrm{Zar}(\Omega^*_X)$ in $\ho{\mathrm{dga}_\C}$. This can now be easily transposed into positive characteristic, since the infinitesimal site has a good analogue in rigid cohomology, the overconvergent site of Le Stum. Thus, we are led to give a second definition of the rigid rational homotopy type, namely as $\R\Gamma(\mathcal{O}^\dagger_{X/K})$, where $\mathcal{O}_{X/K}^\dagger$ is the constant crystals on the overconvergent site, and $\R\Gamma$ is taken in the sense of dga's.

\subsection{The overconvergent site}

We now recall the definition of Le Stum's overconvergent site, and give a new definition of the rational homotopy type. The main reference is \cite{LS11}. We will systematically consider analytic spaces in the sense of Berkovich, and we will call an analytic variety over $K$ a locally Hausdorff, good, strictly $K$-analytic space. If $V$ is an analytic variety, then we will denote by $V_0$ the underlying rigid space, and $\pi_V:V_0\rightarrow V$ the natural map. If $\mathscr{P}$ is a formal $\mathcal{V}$-scheme, then $\mathscr{P}_K$ will denote its Berkovich generic fibre, and $\mathscr{P}_{K0}$ its rigid generic fibre. (This is the reason for putting $0$'s everywhere in the previous section). Recall that a Berkovich space is called good if every point has an affinoid neighbourhood.

\begin{definition} An overconvergent variety over $\mathcal{V}$ consists of the data of a $k$-variety $X$, a formal $\mathcal{V}$-scheme $\mathscr{P}$ and an analytic $K$-variety $V$, together with an embedding $X\subset \mathscr{P}$ of formal $\mathcal{V}$-schemes, and a morphism $\lambda:V\rightarrow \mathscr{P}_K$ of Berkovich spaces. An overconvergent variety will often be denoted $(X\subset \mathscr{P}\leftarrow V)$. A morphism of overconvergent varieties is a commutative diagram 
\begin{equation}
\xymatrix{
X' \ar[d]^{f} \ar@{^{(}->}[r] &\mathscr{P}'\ar[d]^{v} & \ar[l]_{\mathrm{sp}} \ar[d]^{v_K} \mathscr{P}'_K & \ar[d]^{u} \ar[l]_{\lambda'} V' \\
X \ar@{^{(}->}[r] &\mathscr{P} & \ar[l]_{\mathrm{sp}}\mathscr{P}_K & \ar[l]_{\lambda} V
}
\end{equation}
and the category of overconvergent varieties over $\mathcal{V}$ is denoted $\mathrm{An}(\mathcal{V})$.
\end{definition}

For $(X\subset \mathscr{P}\leftarrow V)$ an overconvergent variety, define the tube $]X[_V=(\mathrm{sp}\circ\lambda)^{-1}(X)\subset V$. Denote by $i_{X,V}:]X[_V\rightarrow V$ the natural inclusion (we will often write $i_X$ instead). A morphism of overconvergent varieties is called a strict neighbourhood if $X=X'$, $\mathscr{P}=\mathscr{P}'$, $u:V'\rightarrow V$ is the inclusion of an open neighbourhood of $]X[_V$ in $V$, and $]X[_{V'}=]X[_V$. In \emph{loc. cit.}, Le Stum proves that the category $\mathrm{An}(\mathcal{V})$ admits calculus of right fractions with respect to strict neighbourhoods, and denotes by $\mathrm{An}^\dagger(\mathcal{V})$ the localized category.

The category $\mathrm{An}(\mathcal{V})$ admits a topology coming from the analytic topology of $V$, and this induces a topology on $\mathrm{An}^\dagger(\mathcal{V})$, called the analytic topology. Since the formal scheme $\mathscr{P}$ plays less of a role in the category $\mathrm{An}^\dagger(\mathcal{V})$, we usually denote objects by, for example, $(X,V)$. The functor $(X,V)\mapsto \Gamma(]X[_V,i_{X,V}^{-1}\mathcal{O}_V)$ is then well defined, and is a sheaf of $\mathrm{An}^\dagger(\mathcal{V})$, denoted $\mathcal{O}_{\mathcal{V}}^\dagger$ and called the sheaf of overconvergent functions.

Fix some object $(C,O)$ of $\mathrm{An}^\dagger(\mathcal{V})$, and consider the restricted category $\mathrm{An}^\dagger(C,O)$ of all objects of $\mathrm{An}^\dagger(\mathcal{V})$ over $(C,O)$. Denote by $j_{C,O}:\mathrm{An}^\dagger(C,O)\rightarrow \mathrm{An}^\dagger(\mathcal{V})$ the corresponding morphism of sites. Le Stum defines a morphism of sites $I_{C,O}:\mathrm{Sch}(C)\rightarrow \mathrm{An}^\dagger(C,O)$ (where $\mathrm{Sch}(C)$ is given the coarse topology), given by $I_{C,O}^{-1}(X,V)=X$.

\begin{definition} Let $X$ be an algebraic variety over $C$, let $\underline{X}$ be the corresponding representable sheaf on the site $\mathrm{Sch}(C)$. Then the sheaf of overconvergent varieties over $X$ above $(C,O)$ is by definition $X/O:={j_{C,O}}_!{I_{C,O}}_*\underline{X}$. The site $\mathrm{An}^\dagger(X/O)$ is the restricted site of objects of $\mathrm{An}^\dagger(\mathcal{V})$ over $X/O$, and the corresponding topos is denoted $(X/O)_{\mathrm{An}^\dagger}$. The restriction of $\mathcal{O}_{\mathcal{V}}^\dagger$ to $\mathrm{An}^\dagger(X/O)$ will be denoted $\mathcal{O}_{X/O}^\dagger$. 
\end{definition}

For any morphism of $C$-varieties $f:X\rightarrow Y$ there is a morphism of sheaves $X/O\rightarrow Y/O$ and hence a morphism of topoi $f_{\mathrm{An}^\dagger}:(X/O)_{\mathrm{An}^\dagger}\rightarrow (Y/O)_{\mathrm{An}^\dagger}$. Since $f_{\mathrm{An}^\dagger}^{-1}(\mathcal{O}_{Y/O}^\dagger)=\mathcal{O}_{X/O}^\dagger$, this naturally becomes a morphism of ringed topoi. Letting $p:X\rightarrow C$ denote the structure morphism of a $C$-variety $X$, the functor $O'\mapsto (C,O')$ defines a morphism of sites $\mathrm{An}^\dagger(C,O)\rightarrow\mathrm{Open}(]C[_O)$, and hence we can consider the composite morphism of topoi $p_{X/O}:(X/O)_{\mathrm{An}^\dagger}\rightarrow (C,O)_{\mathrm{An}^\dagger}\rightarrow ]C[_O^{\mathrm{an}}$. 

If $(C,O)$ is an overconvergent variety, then we will denote by $(-)^\mathrm{an}$ the derived push-forward functor $\R\pi_*$ for the morphism of ringed spaces 
\begin{equation}
\pi:(]\overline{C}[_{O0},j^\dagger\mathcal{O}_{]\overline{C}[_{O0}})\rightarrow (]\overline{C}[_{O}, i_{C*}i_C^{-1}\mathcal{O}_{]\overline{C}[_{O}})
\end{equation}
where $\overline{C}$ denote the closure of $C$ inside the `unmentioned' formal scheme of $(C,O)$. The main results of \emph{loc. cit.} are the following. 

\begin{theorem} \label{maincomp}(\cite{LS11}, Theorem 3.6.7). Let $\mathscr{S}$ be a good formal $\mathcal{V}$-scheme, and consider the object $(\mathscr{S}_k,\mathscr{S}_K)$ of $\mathrm{An}^\dagger(\mathcal{V})$. Let $X$ be an algebraic variety over $\mathscr{S}_k$, with structure morphism $p$. 
\begin{enumerate} \item There is a canonical equivalence between the category of finitely presented $\mathcal{O}_{X/\mathscr{S}_K}^\dagger$-modules and the category of overconvergent isocrystals on $X/\mathscr{S}$.
\item For any overconvergent isocrystal $E$ on $X/\mathscr{S}$, there is an isomorphism $(\mathbb{R}p_{X/\mathscr{S},\mathrm{rig}*} E)^\mathrm{an} \cong \mathbb{R}p_{X/\mathscr{S}_K*}E$.
\end{enumerate} 
\end{theorem}

When $\mathscr{S}=\mathrm{Spf}(\mathcal{V})$, we will often write $\Gamma$ instead of $p_{X/K*}$, thus for a finitely presented $\mathcal{O}_{X/K}^\dagger$-module $E$, the above result becomes an isomorphism 
\begin{equation}
H^i_\mathrm{rig}(X/K,E)\cong \R^i\Gamma((X/O)_{\mathrm{An}^\dagger},E). 
\end{equation}  
However, this result is not quite enough for our purposes, we want to be able to take any smooth triple $(S,\overline{S},\mathscr{S})$, which is not accounted for in Le Stum's comparison theorem. However, this extension is straightforward.

\begin{theorem} \label{extendedcomp}Let $(S,\overline{S},\mathscr{S})$ be a smooth triple, with $\mathscr{S}$ good, and $p:X\rightarrow S$ a morphism of $k$-varieties. Let $E$ be an overconvergent isocrystal on $(X/\mathscr{S})$. Let $i_S:]S[_\mathscr{S}\rightarrow ]\overline{S}[_\mathscr{S}$ denote the (closed) inclusion. Then there is a quasi-isomorphism \begin{equation}
 i_S^{-1} ( \R p_{X/\mathscr{S},\mathrm{rig}*}E)^\mathrm{an} \cong (\R p_* E)_{(S,\mathscr{S}_K)}. 
\end{equation}
\end{theorem}

\begin{proof} It suffices to show that $( \R p_{X/\mathscr{S},\mathrm{rig}*}E)^\mathrm{an} \cong i_{S*}(\R p_* E)_{(S,\mathscr{S}_K)}$, and the proof of this is virtually word for word the same as in the proof of Theorem 3.6.7 of \cite{LS11}, taking care that in the proof of Proposition 3.5.8 of \emph{loc. cit.} one must replace the analytic space $\mathscr{S}_K$ by $]\overline{S}[_{\mathscr{S}}$ (and similarly in the rigid case) and the equality $\R v_{K*}\circ i_{X*}=\R v_{K*}$ by the equality $\R v_{K*}\circ i_{X*}=i_{S*}\circ \R v_{K*}$  \end{proof}

Let $\mathrm{dga}(\mathcal{O}_{X/K}^\dagger)$ denote the category of sheaves of $\mathcal{O}_{X/K}^\dagger$-dga's. If $f:X\rightarrow Y$ is a morphism of $k$-varieties, then since $f^{-1}(\mathcal{O}_{Y/K}^\dagger)=\mathcal{O}_{X/K}^\dagger$, the functor $f_*:\mathrm{dga}(\mathcal{O}_{X/K}^\dagger)\rightarrow \mathrm{dga}(\mathcal{O}_{Y/K}^\dagger)$ is right Quillen, and we can consider the derived functor 
\begin{equation}
\R f_*:\ho{\mathrm{dga}(\mathcal{O}_{X/K}^\dagger)}\rightarrow \ho{\mathrm{dga}(\mathcal{O}_{Y/K}^\dagger)} 
\end{equation}
as well as the absolute version 
\begin{equation}
\R \Gamma:\ho{\mathrm{dga}(\mathcal{O}_{X/K}^\dagger)}\rightarrow \ho{\mathrm{dga}_K}.
\end{equation}
The definition of the rational homotopy type of a $k$-variety is now straightforward. 

\begin{definition} The rational homotopy type of $X$ is $\R\Gamma(\mathcal{O}_{X/K}^\dagger)\in\ho{\mathrm{dga}_K}$. If $f:X\rightarrow Y$ is a morphism of $k$-varieties, then the relative rational homotopy type of $X$ over $Y$ is $\R f_*(\mathcal{O}_{X/K}^\dagger)\in\ho{\mathrm{dga}(\mathcal{O}_{Y/K}^\dagger)}$.
\end{definition}

It is easy to check that the rational homotopy type is functorial, and we get a map \begin{equation}
\mathrm{Mor}_{\mathrm{Sch}/k}(X,Y)\rightarrow [\R\Gamma(\mathcal{O}_{Y/K}^\dagger),\R\Gamma(\mathcal{O}_{X/K}^\dagger)]_{\ho{\mathrm{dga}_K}}.
\end{equation}
Similarly, for every morphism $f:X\rightarrow Y$ there is a map from the set of sections of $f$ to the set of sections (taking care with contravariance!) of the induced map $\R\Gamma(\mathcal{O}_{Y/K}^\dagger)\rightarrow \R\Gamma(\mathcal{O}_{X/K}^\dagger)$ in $\ho{\mathrm{dga}_K}$. There is also the obvious relative version of this.

\subsection{A comparison theorem}

In this section, we prove the following comparison result.

\begin{theorem} \label{comp}There is an isomorphism 
\begin{equation}
 \R\Gamma_\mathrm{Th}(\Omega^*(\mathcal{O}_{X/K}^\dagger))\cong \R\Gamma(\mathcal{O}_{X/K}^\dagger)
\end{equation}
in $\ho{\mathrm{dga}_K}$.
\end{theorem}

The idea is that after using simplicial methods to (essentially) reduce to the case where we may choose an embedding of $X$ into a smooth formal $\mathcal{V}$-scheme, we only really need to observe that Le Stum's comparison of rigid cohomology and cohomology of the overconvergent site respects multiplicative structures.

So suppose that $\mathfrak{U}_\bu=(U_\bu,\overline{U}_\bu,\mathscr{U}_\bu)$ is a framing system for $X$, with $U_\bu\rightarrow X$ a Zariski hyper-covering. Then, we define the category $\mathrm{dga}(\mathcal{O}_{U_\bu/K}^\dagger)$ of dga's on the simplicial ringed topos $(U_\bu/K)_{\mathrm{An}^\dagger}$ in the standard way. As before, we can consider the functor of Thom--Whitney global sections
\begin{equation}
\R\Gamma_\mathrm{Th}=\mathrm{Th}\circ\R\Gamma:\ho{\mathrm{dga}(\mathcal{O}_{U_\bu/K}^\dagger)}\rightarrow \ho{\mathrm{dga}_K}.
\end{equation}
There is also an obvious restriction functor $(-)|_{U_\bu}:\mathrm{dga}(\mathcal{O}_{X/K}^\dagger)\rightarrow \mathrm{dga}(\mathcal{O}_{U_\bu/K}^\dagger)$, thus giving us two functors 
\begin{equation}
\xymatrix{ \ho{\mathrm{dga}(\mathcal{O}_{X/K}^\dagger)} \ar@/^1.5pc/[rrr]^{\R\Gamma_\mathrm{Th}\circ (-)|_{U_\bu}} \ar@/_1.5pc/[rrr]_{\R\Gamma} & & & \ho{ \mathrm{dga}_K}.} 
\end{equation}
which we wish to compare. Note that there is an obvious natural transformation $\R\Gamma\Rightarrow \R\Gamma_\mathrm{Th}\circ (-)|_{U_\bu}$.

\begin{proposition} This natural transformation is an isomorphism when evaluated on $\mathcal{O}_{X/K}^\dagger$.
\end{proposition}

\begin{proof} As usual, it suffices to show that it induces an isomorphism on cohomology. But this just follows from cohomological descent for overconvergent cohomology, see Section 3.6 of \cite{LS11}.
\end{proof}

We now want to extend Le Stum's overconvergent version of `linearization of differential operators' to deal both with dga's and with simplicial Berkovich spaces. To start with, consider the following diagram of simplicial ringed topoi 
\begin{equation}
\xymatrix{ (U_\bu,]U_\bu[_{\mathscr{U}_\bu})_{\mathrm{An}^\dagger} \ar[r]^-{\varphi_{\mathfrak{U}_\bu}}\ar[d]_{j_{\mathfrak{U}_\bu}} & ]U_\bu[_{\mathscr{U}_\bu} \\  (U_\bu/K)_{\mathrm{An}^\dagger} } 
\end{equation} 
where:

\begin{itemize}
\item $(U_\bu/K)_{\an}$ is as described above, $]U_\bu[_{\mathscr{U}_\bu}$ is the tube associated to $(U_\bu,\mathscr{U}_\bu)$ and $(U_\bu,]U_\bu[_{\mathscr{U}_\bu})_{\an}$ is the simplicial topos of sheaves on $\mathrm{An}^\dagger(\mathcal{V})$ over the representable simplicial sheaf associated to $(U_\bu,]U_\bu[_{\mathscr{U}_\bu})$;
\item $j_{\mathfrak{U}_\bu}$ arises from the natural morphism $(U_\bu,]U_\bu[_{\mathscr{U}_\bu})\rightarrow (U_\bu/K)$ of simplicial sheaves, and $\varphi_{\mathfrak{U}_\bu}$ is the `realization map'. For more details, see \S\S1.4 and 2.1 of \emph{loc. cit.}
\end{itemize}

The induced maps $j_{\mathfrak{U}_\bu}^{-1}(\mathcal{O}^\dagger_{U_\bu/K})\rightarrow \mathcal{O}^\dagger_{U_\bu/]U_\bu[_{\mathscr{U}_\bu}}$ and $\varphi_{\mathfrak{U}_\bu}^{-1}(K) \rightarrow \mathcal{O}^\dagger_{U_\bu/]U_\bu[_{\mathscr{U}_\bu}}$ are both flat, and we can thus consider the derived linearization of the overconvergent de Rham dga 
\begin{equation}
\R L(i_{U_\bu}^{-1}\Omega^*_{(\mathscr{U}_\bu)_K}):=\R j_{\mathfrak{U}_\bu*} \varphi_{\mathfrak{U}_\bu}^*(i_{U_\bu}^{-1}\Omega^*_{(\mathscr{U}_\bu)_K}) 
\end{equation} 
as in Chapter 3 of \emph{loc. cit.}. This is an object of $\ho{\mathrm{dga}(\mathcal{O}_{U_\bu/K}^\dagger)}$. 

\begin{proposition} There is an isomorphism
\begin{equation}
\mathcal{O}_{U_\bu/K}^\dagger\rightarrow  \R L(i_{U_\bu}^{-1}\Omega^*_{(\mathscr{U}_\bu)_K})
\end{equation}
in $\ho{\mathrm{dga}(\mathcal{O}_{U_\bu/K}^\dagger)}$.
\end{proposition}

\begin{proof} To define the morphism, it suffices to define a morphism  
\begin{equation}
 \mathcal{O}_{U_\bu/K}^\dagger\rightarrow j_{\mathfrak{U}_\bu*}\varphi_{\mathfrak{U}_\bu}^*(i_{U_\bu}^{-1}\Omega^*_{(\mathscr{U}_\bu)_K})
\end{equation}
in $\mathrm{dga}(\mathcal{O}_{U_\bu/K}^\dagger)$, or equivalently a map 
$ \mathcal{O}_{U_n/K}^\dagger \rightarrow j_{\mathfrak{U}_n*}\varphi_{\mathfrak{U}_n}^*(i_{U_n}^{-1}\mathcal{O}_{(\mathscr{U}_n)_K})$
of $K$-algebras, functorially in $n$, such that the composite map 
$\mathcal{O}_{U_n/K}^\dagger \rightarrow j_{\mathfrak{U}_n*}\varphi_{\mathfrak{U}_n}^*(i_{U_n}^{-1}\Omega^1_{(\mathscr{U}_n)_K})
$
is zero. But exactly as in Proposition 3.3.10 of \cite{LS11}, since $\mathcal{O}^\dagger_{U_n/K}$ is a crystal, we have $j_{\mathfrak{U}_\bu}^{-1}\mathcal{O}_{U_\bu/K}^\dagger\cong\varphi_{\mathfrak{U}_\bu}^*(i_{U_\bu}^{-1}\Omega^*_{(\mathscr{U}_\bu)_K})$, and hence this map arises via the adjunction between $j_{\mathfrak{U}_\bu*}$ and $j_{\mathfrak{U}_\bu}^{-1}$. To prove that the induced map $\mathcal{O}_{U_\bu/K}^\dagger\rightarrow  \R L(i_{U_\bu}^{-1}\Omega^*_{(\mathscr{U}_\bu)_K})$ is a quasi-isomorphism, we can forget the algebra structure, and prove that it is a quasi-isomorphism of complexes in each simplicial degree. But by definition, the map $\mathcal{O}_{U_n/K}^\dagger\rightarrow \R L(i_{U_\bu}^{-1}\Omega^*_{(\mathscr{U}_\bu)_K})_n$ is exactly the augmentation map that Le Stum constructs. That this is a quasi-isomorphism is then Proposition 3.5.4. of \emph{loc. cit}.
\end{proof}

\begin{proposition} There is an isomorphism 
\begin{equation}
\R\Gamma(\R L(i_{U_\bu}^{-1}\Omega^*_{(\mathscr{U}_\bu)_K}))\cong \R\Gamma(i_{U_\bu}^{-1}\Omega^*_{(\mathscr{U}_\bu)_K}))
\end{equation} 
in $\ho{\mathrm{dga}_K^\Delta}$.
\end{proposition}

\begin{proof} Just note that the proof of Proposition 3.3.9 of \emph{loc. cit.} carries over \emph{mutatis mutandis} to the simplicial/dga situation.
\end{proof}

Combining these two results, we see that in order to complete the proof of Theorem \ref{comp}, we just need to verify that there is a canonical isomorphism 
\begin{equation}
 \R\Gamma(j^\dagger\Omega^*_{V_0(\mathfrak{U})_\bu})\cong \R\Gamma(i_{U_\bu}^{-1}\Omega^*_{(\mathscr{U}_\bu)_K})
\end{equation}
in $\ho{\mathrm{dga}_K^\Delta}$. We may work level-wise, where there is a natural map 
\begin{equation}
\R\Gamma(j^\dagger\Omega^*_{V_0(\mathfrak{U})_n})\rightarrow \R\Gamma(i_{U_n}^{-1}\Omega^*_{(\mathscr{U}_n)_K})
\end{equation}
which comes from the map of topoi $V_0(\mathfrak{U})_n\rightarrow (\mathscr{U}_n)_K$ and the comparison between $j^\dagger\Omega^*_{]\overline{U}_n[_{\mathscr{U}_n0}}$ and $i_{U_n}^{-1}\Omega^*_{(\mathscr{U}_n)_K}$, as in Proposition 3.4.3 of \cite{LS11}. To show that it is an isomorphism is we may forget the algebra structure, and invoke Le Stum's results from \S3 of \emph{loc. cit.}

\subsection{Functoriality and Frobenius structures}

We are now in a position to prove that the rigid rational homotopy type is functorial. Indeed, it is clear that the definition in terms of the overconvergent site is functorial, and it is also not too difficult to see by functoriality of the comparison morphism that the map $f^*:\mathbb{R}\Gamma(\mathcal{O}^\dagger_{Y/K})\rightarrow \mathbb{R}\Gamma(\mathcal{O}^\dagger_{X/K})$ induced by any morphism $f:X\rightarrow Y$ is the same as that induced by any lift of $f$ to a map $\mathfrak{f}$ between framing systems for $X$ and $Y$. In particular, this latter map is independent of the lift $\mathfrak{f}$.

In order to put Frobenius structures on the dga's obtained from the overconvergent site, we need to examine slightly more closely Le Stum's base change morphism 1.4.6 of \cite{LS11}. So suppose that $\alpha:K\rightarrow K'$ is a finite extension of complete, discretely valued fields, and let $\mathcal{V}\rightarrow \mathcal{V}'$ (respectively, $k\rightarrow k'$) be the induced finite extension of rings of integers (respectively, residue fields). Then, there is a morphism of sites 
\begin{equation}
 \alpha: \mathrm{An}^\dagger(\mathcal{V}')\rightarrow \mathrm{An}^\dagger(\mathcal{V}) 
\end{equation}
which is induced by $(X\subset \mathscr{P}\leftarrow V)\mapsto (X_{k'}\subset \mathscr{P}_{\mathcal{V}'}\leftarrow V_{K'})$. Important for us will be the fact that this base extension functor has an adjoint, which considers an overconvergent variety $(Y,W)$ over $K'$ as one over $K$ - note that this holds only if the extension $K\rightarrow K'$ is finite. Hence the pull-back morphism $\alpha^{-1}$ on presheaves has a simple description - namely $(\alpha^{-1}\mathcal{F})(Y,W)=\mathcal{F}(Y,W)$ where on the LHS we are considering $(Y,W)$ as an overconvergent variety over $K'$, and on the RHS as one over $K$. In particular, we have $\alpha^{-1}(\mathcal{O}_{\mathcal{V}}^\dagger)=\mathcal{O}^\dagger_{\mathcal{V}'}$, and $\alpha$ extends to a morphism of ringed sites.

Now suppose that $X$ is a $k$-variety, so that we have the sheaf $(X/K)$ on $\mathrm{An}^\dagger(\mathcal{V})$, which is the sheafification of the presheaf $(C,O)\mapsto \mathrm{Mor}_k(C,X)$. By the above comments, $\alpha^{-1}(X/K)$ is the sheafification of the presheaf $(C',O')\mapsto \mathrm{Mor}_k(C',X)=\mathrm{Mor}_{k'}(C',X_{k'})$. Thus, we see that $\alpha^{-1}(X/K)=(X_{k'}/K')$, and hence we get a morphism of ringed topoi 
\begin{equation}
(X_{k'}/K')_{\mathrm{An}^\dagger}\rightarrow (X/K)_{\mathrm{An}^\dagger}.
\end{equation}
More generally, exploiting functoriality of $(Y/K')_{\mathrm{An}^\dagger}$ in $Y$ as a $k'$-variety, we see that for any $k'$ variety $Y$ and any commutative square 
\begin{equation}
 \xymatrix{ Y\ar[r]^f \ar[d] & X\ar[d] \\ \spec{k'}\ar[r] & \spec{k} }
\end{equation}
there is an induced morphism of ringed topoi 
\begin{equation}
f: (Y/K')_{\mathrm{An}^\dagger} \rightarrow (X/K)_{\mathrm{An}^\dagger}
\end{equation}
such that that $f^{-1}(\mathcal{O}_{X/K}^\dagger)=\mathcal{O}_{Y/K'}^\dagger$. The situation we are interested in is when $\sigma:K\rightarrow K$ is a lifting of the absolute Frobenius on $k$, and $F_X:X\rightarrow X$ is the absolute Frobenius on $X$. We then get a morphism 
\begin{equation}
F_X:(X/K)_{\mathrm{An}^\dagger}\rightarrow (X/K)_{\mathrm{An}^\dagger} 
\end{equation}
of ringed topoi, and if $f: X\rightarrow Y$ is a morphism of $k$-varieties, then there is a commutative square 
\begin{equation}
 \xymatrix{ (X/K)_{\an}\ar[r]^{F_X}\ar[d]_f & (X/K)_{\an}\ar[d]^f \\ (Y/K)_{\an}\ar[r]^{F_Y}&(Y/K)_{\an}. }
\end{equation}
Hence, we get a base change map 
\begin{equation}
\Phi_{X/Y}:F_Y^{-1}\R f_*(\mathcal{O}_{X/K}^\dagger) \rightarrow \R f_*(\mathcal{O}_{X/K}^\dagger)
\end{equation}
in $\ho{\mathrm{dga}(\mathcal{O}_{Y/K}^\dagger)}$.

\begin{proposition} If $Y=\spec{k}$ is a point, then $\Phi_{X/k}$ is a quasi-isomorphism.
\end{proposition}

\begin{proof} This is a straightforward application of the comparison theorem. It is not too difficult to check that this is compatible base change, and hence the morphism induced by $\Phi_{X/k}$ on cohomology is the usual Frobenius on rigid cohomology, which is an isomorphism.
\end{proof}

\begin{remark} Similarly to the problem of functoriality, we can see that the map $\Phi_{X/k}$ is the same as the map induced by a lift of the absolute Frobenius to a framing system for $X$. Again, this implies that the latter is independent of the choice of this lift.
\end{remark}

\begin{remark} We do not know whether or not $\Phi_{X/Y}$ is a quasi-isomorphism in general. It would follow, for example, if we knew Frobenius to be bijective on relative overconvergent cohomology.
\end{remark}

\section{Relative crystalline homotopy types}

In this section, we define relative rational homotopy types, and again there will be two approaches, one via rigid cohomology and cohomological descent and one via the overconvergent site of Le Stum.

In rigid cohomology, the relative theory is expressed with respect to a base frame. We will also systematically work with pairs of varieties over $k$, that is, we will work in the category consisting of open immersions $S\rightarrow\overline{S}$ of $k$-varieties, and where morphisms are commutative diagrams. The reason we do this is to more easily apply the results of \cite{CT03} on cohomological descent.

Fix a base frame $\mathfrak{S}=(S,\overline{S},\mathscr{S})$, which we assume to be smooth and proper over $\mathcal{V}$. Although most of what we say will work in greater generality, we will be mainly interested in the case where $S$ is a smooth, geometrically connected curve over $k$, $\overline{S}$ is its unique compactification, and $\mathscr{S}$ is a lifting of $\overline{S}$ to a smooth formal curve over $\mathcal{V}$.

\begin{definition} We say that a frame $\mathfrak{U}=(U, \overline{U},\mathscr{U})$ over $\mathfrak{S}$ is smooth if $\mathscr{U}\rightarrow \mathscr{S}$ is smooth in a neighbourhood of $U$, and proper if $\mathscr{U}\rightarrow \mathscr{S}$ is.
\end{definition}

\begin{definition} Let $(X,\overline{X})$ be a pair of varieties over $k$, that is, an open immersion of separated $k$-schemes of finite type. Let $f:(X,\overline{X})\rightarrow (S,\overline{S})$ be a morphism of pairs. Then, an $(X,\overline{X})$-frame over $\mathfrak{S}$ is a frame $\mathfrak{Y}=(Y,\overline{Y},\mathscr{Y})$ over $\mathfrak{S}$ together with a morphism $(Y,\overline{Y})\rightarrow (X,\overline{X})$ such that the diagram
\begin{equation}
\xymatrix{  (Y,\overline{Y})\ar[r]\ar[dr] & (X,\overline{X})\ar[d] \\ & (S,\overline{S}) }
\end{equation}  
commutes.
\end{definition}

\begin{definition} Let $f:(X,\overline{X})\rightarrow (S,\overline{S})$ be as above. Then, we define a framing system for $f$ to be simplicial $(X,\overline{X})$-frame $\mathfrak{Y}_\bullet=(Y_\bullet,\overline{Y}_\bullet,\mathscr{Y}_\bullet)$ over $\mathfrak{S}$, such that each $\mathfrak{Y}_n$ is smooth over $\mathfrak{S}$, and which is universally de Rham descendable, in the sense of \cite{CT03}, Definition 10.1.3.
\end{definition}

Of course, the definition is rigged exactly so that we can apply Chiarellotto and Tsuzuki's theory of cohomological descent for relative rigid cohomology. Since we are really interested in the case of a morphism $X\rightarrow S$, we need to check that we are not unduly restricting the scope of our theory.

\begin{proposition} Suppose that $X\rightarrow S$ is a morphism of $k$-varieties. Then there exists a a pair $(X,\overline{X})$ and a morphism of pairs $f:(X,\overline{X})\rightarrow (S,\overline{S})$ such that $\overline{X}$ is proper over $S$ and $f$ admits a framing system.
\end{proposition}

\begin{proof} That there exists a proper $\overline{S}$-scheme $\overline{X}$ and a morphism of pairs $f:(X,\overline{X})\rightarrow(S,\overline{S})$ as claimed is Nagata's compactification theorem. 

By Example 6.1.3, (1) of \cite{Tsu04}, it suffices to show that there exists a Zariski covering of $(X,\overline{X})$ over $\mathfrak{S}$, that is, an $(X,\overline{X})$ frame $\mathfrak{U}=(U,\overline{U},\mathscr{U})$ which is smooth over $\mathfrak{S}$, such that $\overline{u}:\overline{U}\rightarrow X$ is an open covering and $U=\overline{u}^{-1}(X)$. Now, since $\overline{X}$ is separated and of finite type over $\spec{k}$, we may choose an open affine cover $\overline{U}_i$ of $\overline{X}$, and a closed embedding $\overline{U}_i\hookrightarrow \A^{n_i}_k$ into some affine space over $k$. We now define $\overline{U}=\CMcoprod_i\overline{U}_i$, $U$ to be the pull-back of $\overline{U}\rightarrow \overline{X}$ to $X$. Since $\overline{X}\rightarrow \overline{S}$ is proper, it is an open mapping onto its (closed) image, and hence we can choose an open subset $\mathscr{S}_i$ of $\mathscr{S}$ such that for each $i$ induced map $\overline{U}_i\rightarrow \widehat{\A}^{n_i}_\mathcal{V}\times_\mathcal{V}\mathscr{S}_i$ is a closed immersion. Thus setting $\mathscr{U}=\CMcoprod_i \widehat{\A}^{n_i}_\mathcal{V}\times_\mathcal{V}\mathscr{S}_i$ gives us the required Zariski cover $(U,\overline{U},\mathscr{U})$ of $(X,\overline{X})$ over $\mathfrak{S}$.
\end{proof}

Now we proceed exactly as in the previous section, simply replacing the frame $\mathrm{Sp}(K)=(\spec{k},\spec{k},\mathrm{Spf}(\mathcal{V}))$ everywhere by $\mathfrak{S}$. If we are given a morphism of pairs $f:(X,\overline{X})\rightarrow (S,\overline{S})$ and a framing system $\mathfrak{f}:\mathfrak{Y}_\bullet\rightarrow \mathfrak{S}$ for $f$, then we get a simplicial space $V_0(\mathfrak{Y}_\bullet):= ]\overline{Y}_\bullet[_{\mathscr{Y}_\bullet0}$ over $]\overline{S}[_{\mathscr{S}0}$. Hence, we can consider the category $\mathrm{dga}(V_0(\mathfrak{Y}_\bullet);j^\dagger\mathcal{O}_{]\overline{S}[_{\mathscr{S}0}})$ of sheaves of $j^\dagger\mathcal{O}_{]\overline{S}[_\mathscr{S}0}$-dga's on the simplicial space $V_0(\mathfrak{Y}_\bullet)$.

Exactly as in the absolute case, we have derived push-forward functors
\begin{align}
\R\mathfrak{f}_{K0*}:\ho{\mathrm{dga}(V_0(\mathfrak{Y}_\bullet);j^\dagger\mathcal{O}_{]\overline{S}[_{\mathscr{S}0}})}&\rightarrow \ho{\mathrm{dga}(]\overline{S}[_{\mathscr{S}0};j^\dagger\mathcal{O}_{]\overline{S}[_{\mathscr{S}0}})^\Delta} \\
\R\mathfrak{f}_{K0*\mathrm{Th}}:=\mathrm{Th}\circ \R\mathfrak{f}_{K0*}:\ho{\mathrm{dga}(V_0(\mathfrak{Y}_\bullet);j^\dagger\mathcal{O}_{]\overline{S}[_{\mathscr{S}0}})}&\rightarrow \ho{\mathrm{dga}(]\overline{S}[_{\mathscr{S}0};j^\dagger\mathcal{O}_{]\overline{S}[_{\mathscr{S}0}})}.
\end{align}
For each $n$ we have the sheaf of $j^\dagger\mathcal{O}_{]\overline{S}[_{\mathscr{S}0}}$-dga's $j^\dagger\Omega^*_{]\overline{Y}_n[_{\mathscr{Y}_n0}/]\overline{S}[_{\mathscr{S}0}} $
which fit together to gives a sheaf of $j^\dagger\mathcal{O}_{]\overline{S}[_{\mathscr{S}0}}$-dga's $j^\dagger\Omega^*_{]\overline{Y}_\bu[_{\mathscr{Y}_\bu0}/]\overline{S}[_{\mathscr{S}0}} $ on $V_0(\mathfrak{Y}_\bu)$.

\begin{definition} We define the relative rigid rational homotopy type to be 
\begin{equation}
 \R f_{*\mathrm{Th}}(\Omega^*(\mathcal{O}_{X/S}^\dagger)):=\R\mathfrak{f}_{K0*\mathrm{Th}}(j^\dagger\Omega^*_{]\overline{Y}_\bu[_{\mathscr{Y}_\bu0}/]\overline{S}[_{\mathscr{S}0}}))\in\ho{\mathrm{dga}(]\overline{S}[_{\mathscr{S}0};j^\dagger\mathcal{O}_{]\overline{S}[_{\mathscr{S}0}})}.
\end{equation}
\end{definition}

As noted above, we may also define the relative rational homotopy type using the functoriality of the overconvergent site. A morphism $f:X\rightarrow S$ of varieties induces a functor
\begin{equation}
 \R f_*:\ho{\mathrm{dga}(\mathcal{O}_{X/K}^\dagger)}\rightarrow \ho{\mathrm{dga}(\mathcal{O}_{S/K}^\dagger)}
\end{equation}
and we define the relative rational homotopy type to be $\R f_*(\mathcal{O}_{X/K}^\dagger)$. This has some advantages over the previous definition, it is obvious that it only depends on $f:X\rightarrow S$ and not on any choice of compactification or framing system, and subject to certain base change results, it will give us a Gauss--Manin connection on the relative homotopy type. However, it is not particularly computable, and in order to do any calculations, we need the first definition.

\subsection{Another comparison theorem}

In this section, we will prove a comparison theorem between the two approached to relative rigid rational homotopy types. Notation will be exactly as above. The realization functor
\begin{align} \mathrm{Mod}(\mathcal{O}_{S/K}^\dagger)&\rightarrow \mathrm{Mod}(i_{S}^{-1}\mathcal{O}_{]\overline{S}[_\mathscr{S}}) \\ E &\mapsto E_{(S,]\overline{S}[_\mathscr{S})} 
\end{align}
is exact, hence extends to a functor 
\begin{align} \ho{\mathrm{dga}(\mathcal{O}_{S/K}^\dagger)}&\rightarrow \ho{\mathrm{dga}(]S[_{\mathscr{S}};i_S^{-1}\mathcal{O} _{]\overline{S}[_\mathscr{S}})} \\ \mathscr{A}^* &\mapsto \mathscr{A}^*_{(S,]\overline{S}[_\mathscr{S})}.
\end{align}
Recall that we have the morphism of topoi $\pi_{]\overline{S}[_{\mathscr{S}}}:]\overline{S}[_{\mathscr{S}0}\rightarrow ]\overline{S}[_\mathscr{S}$, and that 
\begin{equation}
 \R\pi_{]\overline{S}[_\mathscr{S}*}(j^\dagger\mathcal{O}_{]\overline{S}[_{\mathscr{S}0}})=\pi_{]\overline{S}[_\mathscr{S}*}(j^\dagger\mathcal{O}_{]\overline{S}[_{\mathscr{S}0}})=i_{S*}i_S^{-1}\mathcal{O}_{]\overline{S}[_\mathscr{S}}.
\end{equation}
\begin{lemma} The induced morphism $\pi_{]\overline{S}[_{\mathscr{S}}}^{-1}(i_{S*}i_S^{-1}\mathcal{O}_{]\overline{S}[_\mathscr{S}})\rightarrow j^\dagger\mathcal{O}_{]\overline{S}[_{\mathscr{S}0}}$ is flat.
\end{lemma}

\begin{proof} After replacing $V_0$ by the $G$-topology on $V$, what we must show is that for $V$ a good analytic variety, $W\subset V$ a closed sub-variety, which is open for the $G$-topology, and $\pi_V:V_G\rightarrow V$ the natural map, the induced morphism $\pi_V^{-1}\pi_{V*}(j_{W_G}^\dagger\mathcal{O}_{V_G})\rightarrow j_{W_G}^\dagger\mathcal{O}_{V_G}$ is flat. But this just follows because for any two $G$-open $U'\subset U$ subsets of $V$, the map $\Gamma(\mathcal{O}_{V_G},U)\rightarrow \Gamma(\mathcal{O}_{V_G},U')$ is flat.
 \end{proof}

Hence, we get an induced functor  
\begin{equation}
i_{S}^{-1}\circ \R \pi_{]\overline{S}[_\mathscr{S}*}:\ho{\mathrm{dga}(]\overline{S}[_{\mathscr{S}0};j^\dagger\mathcal{O}_{]\overline{S}[_{\mathscr{S}0}})}\rightarrow \ho{\mathrm{dga}(]S[_\mathscr{S};i_S^{-1}\mathcal{O}_{]\overline{S}[_\mathscr{S}})}
\end{equation}
and we have the following comparison theorem.

\begin{theorem} There is a natural isomorphism  
\begin{equation}
\R f_*(\mathcal{O}_{X/K}^\dagger)_{(S,]\overline{S}[_\mathscr{S})}\rightarrow i_{S}^{-1}\circ \R \pi_{]\overline{S}[_\mathscr{S}*}(\mathrm{Th}(\R \mathfrak{f}_{K0*}(j^\dagger\Omega^*_{]\overline{Y}_\bu[_{\mathscr{Y}_{\bu }0}/]\overline{S}[_{\mathscr{S}0}})))
\end{equation}
in $\ho{\mathrm{dga}(]S[_\mathscr{S};i_S^{-1}\mathcal{O}_{]\overline{S}[_\mathscr{S}})}$.
\end{theorem}

\begin{proof} The proof is almost word for word the same as in the absolute case, taking into account the corresponding statement for cohomology, which is Theorem \ref{extendedcomp}, and its proof, which is essentially contained in Chapter 3 of \cite{LS11}.
\end{proof}

\begin{remark} The comparison theorem can be easily extended to take Frobenius structures into account.
\end{remark}

\subsection{Crystalline complexes and the Gauss--Manin connection} One of the advantages of a `crystalline' definition of the relative rational homotopy type is in the interpretation of the Gauss--Manin connection. By deriving the notion of a crystal, we arrive at a sensible definition of what it means for a complex, or dga, to be crystalline, and the existence of the Gauss--Manin connection is essentially equivalent to $\R f_*(\mathcal{O}_{X/K}^\dagger)$ being crystalline. Unfortunately, at the moment, we cannot prove that this is the case, we can only show that it would follow from a certain `generic coherence' result, for which we give some evidence.

\begin{definition} Suppose that $\mathcal{E}$ is a complex of $\mathcal{O}_{X/K}^\dagger$-modules.
\begin{enumerate}
\item We say that $\mathcal{E}$ is quasi-bounded above if each realisation $\mathcal{E}_{(Y,V)}$ is bounded above.
\item We say that $\mathcal{E}$ is crystalline if it is quasi-bounded above, and for each morphism $u:(Z,W)\rightarrow (Y,V)$ of overconvergent varieties over $(X/K)$, the induced map $\mathbb{L}u^\dagger\mathcal{E}_{(Y,V)}\rightarrow \mathcal{E}_{(Z,W)}$ is an isomorphism in $D^{-}(i_Z^{-1}\mathcal{O}_W)$. Note that this makes sense by the boundedness condition.
\end{enumerate}
An $\mathcal{O}_{X/K}^\dagger$-dga $\mathscr{A}^*$ is said to be crystalline if the underlying complex is crystalline.
\end{definition}

As note above, the reason that we are interested in crystalline dga's is that they give a good interpretation of the Gauss--Manin connection, as we now explain.

Suppose that we have a morphism of $k$-varieties $f:X\rightarrow S$ as above, and a smooth and proper triple $(S,\overline{S},\mathscr{S})$, and that we can show that $\R f_*(\mathcal{O}_{X/K}^\dagger)$ is crystalline. Let $p_i:]S[_{\mathscr{S}^2}\rightarrow ]S[_{\mathscr{S}}$ denote the two natural projections. Then the crystalline nature of $\R f_*(\mathcal{O}_{X/K}^\dagger)\in\ho{\mathrm{dga}(\mathcal{O}^\dagger_{S/K})}$ together with flatness of the $p_i$ means that we have natural quasi-isomorphisms of dga's
\begin{align}
p_1^\dagger \R f_*(\mathcal{O}_{X/K}^\dagger)_{(S,\mathscr{S}_K)}&\rightarrow \R f_*(\mathcal{O}_{X/K}^\dagger)_{(S,\mathscr{S}_K^2)} \\
p_2^\dagger \R f_*(\mathcal{O}_{X/K}^\dagger)_{(S,\mathscr{S}_K)}&\rightarrow \R f_*(\mathcal{O}_{X/K}^\dagger)_{(S,\mathscr{S}_K^2)}
\end{align}
and hence we get an isomorphism 
\begin{equation}
p_1^\dagger \R f_*(\mathcal{O}_{X/K}^\dagger)_{(S,\mathscr{S}_K)}\rightarrow p_2^\dagger \R f_*(\mathcal{O}_{X/K}^\dagger)_{(S,\mathscr{S}_K)}
\end{equation}
in $\ho{\mathrm{dga}(]S[_{\mathscr{S}},i_{S}^{-1}\mathcal{O}_{\mathscr{S}_K^2})}$. In other words, we have a Gauss--Manin connection on the realization $\R f_*(\mathcal{O}_{X/K}^\dagger)_{(S,\mathscr{S}_K)}$, which we can transport over into the rigid world using the comparison theorem between rigid and overconvergent relative rational homotopy types. 

\begin{proposition} \label{crys} Assume that there is some $U\subset Y$ open such that every $\R^qf_*(\mathcal{O}_{X/K}^\dagger)|_U$ is a finitely presented crystal. Then $\R f_*(\mathcal{O}_{X/K}^\dagger)$ is a crystalline dga.
\end{proposition}

Of course, this is really a statement about complexes, rather than dga's. We first show that $\R f_*(\mathcal{O}^\dagger_{X/K})$ is quasi-bounded above.

\begin{lemma} \label{qba} Let $(C,O)$ be an overconvergent variety, and $p:X\rightarrow C$ a $k$-variety over $C$. Then the complex $\R p_{X/O*}(\mathcal{O}_{X/O}^\dagger)\in D^+(i_C^{-1}\mathcal{O}_O)$ is bounded above.
\end{lemma}

\begin{proof} By the spectral sequence associated to a finite open covering (Corollary 3.6.4 of \cite{LS11}), we may assume that $X$ is affine, and hence $p$ has a geometric realization $(X,V)\rightarrow (C,O)$. In fact, we may choose a realization of the following form. We let $(C\hookrightarrow \mathscr{S}\leftarrow O)$ be a triple representing $(C,O)$, and we choose an embedding $X\hookrightarrow \mathscr{P}$ of $X$ into a smooth and proper formal $\mathcal{V}$-scheme. Then, a geometric realization of $X\rightarrow C$ is given by
\begin{equation}
\xymatrix{ X\ar[r]\ar[d] & \mathscr{P}\times_\mathcal{V}\mathscr{S} \ar[d] & V= \mathscr{P}_K\times_K O \ar[l]\ar[d] \\ C \ar[r] & \mathscr{S} & O.\ar[l]  }
\end{equation}
By Theorem 3.5.3 of \emph{loc. cit.}, we must show that $\R p_{]X[_V*}(i_X^{-1}\Omega^*_{V/O})$ is bounded above. Since each term is a coherent $i_X^{-1}\mathcal{O}_V$-module, by the usual spectral sequence relating the cohomology of the complex to the cohomology of each term, it will suffice to show that $\R p_{]X[_V*}$ sends coherent $i_X^{-1}\mathcal{O}_V$-modules to complexes which are bounded above. In fact, we will show that the functor $p_{]X[_V*}$ is exact on coherent $i_X^{-1}\mathcal{O}_V$-modules, which will certainly suffice. 

The question is local on $O$, which we may therefore assume to be affinoid (recall that all our analytic varieties are good). I claim that in this situation, the functor 
\begin{equation}
\mathrm{Coh}(i_X^{-1}\mathcal{O}_V)\rightarrow \mathrm{Mod}_{\mathrm{fp}}(\Gamma(]X[_V,i_X^{-1}\mathcal{O}_V))
\end{equation}
is an equivalence of categories, this is because $]X[_V$ has a cofinal system of neighbourhoods which are all affinoid, and we can apply Proposition 2.2.10 of \cite{LS11} together with the usual result for affinoids. 

Now suppose that $E\overset{\alpha}{\twoheadrightarrow} F$ is a surjection of coherent $i_X^{-1}\mathcal{O}_V$-modules, and consider $G=\mathrm{coker}(p_{]X[_V*}\alpha)$. Then by the above equivalence of categories, $G$ has no non-zero global sections. Moreover, for each affinoid $O'\subset O$, we can apply the same logic to show that $G$ has no global sections when pulled back to $]C[_{O'}$. Hence $G$ is zero, and $p_{]X[_V*}$ is exact for coherent modules, as claimed.

\end{proof}

\begin{corollary} Let $f:X\rightarrow Y$ be a morphism of $k$-varieties. Then $\R f_*(\mathcal{O}_{X/K}^\dagger)$ is quasi-bounded above.
\end{corollary}

\begin{proof} Combine the above proposition with Proposition 3.5.2 of \cite{LS11}.
\end{proof}

\begin{definition}[(\cite{LS11}, Definition 3.6.1)] A complex of $\mathcal{O}_{Y/K}^\dagger$-modules $E$ is said to be of Zariski type if for any overconvergent variety $(C,O)$ over $(Y/K)$, and any open $U\subset O$, with corresponding closed immersion $i:]U[_O\rightarrow ]C[_O$, we have a quasi-isomorphism $i^{-1}E_{(C,O)}\simeq E_{(U,O)} $.
\end{definition}

\begin{remark} Note that the corresponding statement is always true for a closed sub-scheme $Z\subset C$, since then the tube $]Z[_O\subset ]C[_O$ is open.
\end{remark}

\begin{lemma} Let $E$ be a quasi-bounded above complex of $\mathcal{O}_{Y/K}^\dagger$-modules of Zariski type. Let $j:U\rightarrow Y$ be an open immersion, with closed complement $i:Z\rightarrow Y$. Then $E$ is crystalline if and only $j^*E$ and $i^*E$ are both crystalline.
\end{lemma}

\begin{proof} Let $g:(C',O')\rightarrow (C,O)$ be a morphism of overconvergent varieties over $(Y/K)$, then letting e.g. $C_U$ denote $C\times_Y U$, we have a diagram 
\begin{equation}
 \xymatrix{ (C'_U,O')\ar[r]^{j'} \ar[d]_{g_U} & (C',O')\ar[d]^g & (C'_Z,O')\ar[l]_{i'}\ar[d]^{g_Z} \\ (C_U,O) \ar[r]^j\ar[d] & (C,O) \ar[d] & (C_Z,O)\ar[l]_i\ar[d] \\ (U/K) \ar[r]^j & (Y/K) & (Z/K)\ar[l]_i }
\end{equation}
and since $]C'[_{O'}$ is covered by $]C'_U[_{O'}$ and $]C'_Z[_{O'}$, to prove that the morphism
\begin{equation}
\mathbb{L}g^\dagger E_{(C,O)}\rightarrow E_{(C',O')}
\end{equation}
is a quasi-isomorphism, it suffices to prove that the two morphisms 
\begin{align}
i'^{-1}\mathbb{L}g^\dagger E_{(C,O)}&\rightarrow i'^{-1}E_{(C',O')} \\
j'^{-1}\mathbb{L}g^\dagger E_{(C,O)}&\rightarrow j'^{-1}E_{(C',O')} 
\end{align}
are quasi-isomorphisms. But now using the hypothesis that $E$ is of Zariski type and that $j^*E$ and $i^*E$ are crystalline, together with 2.3.2 of \cite{LS11}, we can calculate
\begin{align} i'^{-1}\mathbb{L}g^\dagger E_{(C,O)} &= \mathbb{L}i'^\dagger\mathbb{L}g^\dagger E_{(C,O)} = \mathbb{L}g_Z^\dagger\mathbb{L}i^\dagger E_{(C,O)} \\
&= \mathbb{L}g_Z^\dagger i^{-1} E_{(C,O)} = \mathbb{L}g_Z^\dagger E_{(Z,O)} \\
&\simeq E_{(Z',O')} =i'^{-1}E_{(C',O')}
\end{align}
and 
\begin{align} j'^{-1}\mathbb{L}g^\dagger E_{(C,O)} &= \mathbb{L}j'^\dagger\mathbb{L}g^\dagger E_{(C,O)} = \mathbb{L}g_U^\dagger\mathbb{L}j^\dagger E_{(C,O)} \\
&= \mathbb{L}g_U^\dagger j^{-1} E_{(C,O)} \simeq \mathbb{L}g_U^\dagger E_{(U,O)} \\
&\simeq E_{(U',O')} \simeq j'^{-1}E_{(C',O')}.
\end{align}
\end{proof}

To apply this to $\R f_*(\mathcal{O}_{X/K}^\dagger)$, we will need the following result.

\begin{lemma} Let $f:X\rightarrow Y$ be a morphism of $k$-varieties. Then $\R f_*(\mathcal{O}_{X/K}^\dagger)$ is of Zariski type.
\end{lemma}

\begin{proof} Choose an overconvergent variety $(C,O)$ over $Y/K$, and let $U\subset C$ be an open subset with corresponding inclusion $i:]U[_O\rightarrow ]C[_O$ of tubes. Note that the question is local on both $C$ and $U$ so we may assume that $U\cong D(f)$ for some $f\in\Gamma(C,\mathcal{O}_C)$. First assume that $X_C$ is affine, and that $X_C\rightarrow C$ is has a geometric realization $\mathbf{g}:(X_C,V)\rightarrow (C,O)$. Then, by the proof of Lemma \ref{qba}, coherent $i_X^{-1}\mathcal{O}_V$-modules are $\R \mathbf{g}_{K*}$-acyclic, so we have
\begin{align}
\R f_*(\mathcal{O}_{X/K}^\dagger)_{(C,O)} &=  \R\mathbf{g}_{K*}(i_{X_C}^{-1}\Omega^*_{V/O})=  \mathbf{g}_{K*}(i_{X_C}^{-1}\Omega^*_{V/O}) \\
\R f_*(\mathcal{O}_{X/K}^\dagger)_{(U,O)} &=  \R\mathbf{g}_{K*}(i_{X_U}^{-1}\Omega^*_{V/O})=  \mathbf{g}_{K*}(i_{X_U}^{-1}\Omega^*_{V/O}).
\end{align}
Write $\mathscr{F}^*=i_{X_C}^{-1}\Omega^*_{V/O}$, and let $i':]X_U[_V\rightarrow ]X_C[_V$ and $\mathbf{g}_K':]X_U[_V\rightarrow ]U[_O$ denote the induced map. So we have a Cartesian square
\begin{equation} \xymatrix{ ]X_U[_V \ar[r]^{i'} \ar[d]_{\mathbf{g}_K'} & ]X_C[_V \ar[d]^{\mathbf{g}_K} \\ ]U[_O \ar[r]^i & ]C[_O 
}
\end{equation}
and we need to show that the base change map
\begin{equation} i^{-1}\mathbf{g}_{K*} \mathscr{F}^*\rightarrow \mathbf{g}_{K*}'i'^{-1}\mathscr{F}^*
\end{equation}
is a quasi-isomorphism. Note that $]U[_O$ is given by $\{ x\in ]C[_O \mid |f(x)|\geq1 \}$, and $]X_U[_V$ by $\{ y\in ]X_C[_V \mid |f(\mathbf{g}_K(x))|\geq1 \}$. Hence for any open set $W$ of $]C[_O$, a cofinal system of open neighbourhoods of $W\cap ]U[_O$ in $W$ is given by $T_\eta:=W\cap \{ x\in]C[_O \mid  |f(x)|>\eta \}$ for $\eta<1$, and a cofinal system of neighbourhoods of $\mathbf{g}_K^{-1}(W)\cap ]X_U[_V$ in $\mathbf{g}_K^{-1}(W)$ is given by $\mathbf{g}_K^{-1}(W)\cap \{y\in]X_C[_V \mid |f(\mathbf{g}_K(x))|>\eta\}=\mathbf{g}_K^{-1}(T_\eta)$ for $\eta<1$. Hence, it follows straight from the definition that $i^{-1}\mathbf{g}_{K*}=\mathbf{g}'_{K*}i'^{-1}$ as required. 

To deal with the general case (i.e. $X$ not necessarily affine), note that by Corollary 2.3.2 of \cite{LS11}, a complex of $\mathcal{O}_{Y/K}^\dagger$ is of Zariski type if and only if its cohomology sheaves are. Thus we can choose an open covering of $X_C$ by $C$-varieties admitting geometric realizations to $(C,O)$, and use the spectral sequence 3.6.4 of \emph{loc. cit.} - we know that all the terms on the $E_1$-page are of Zariski type, and hence the abutment must be of Zariski type.
\end{proof}

Now to complete the reduction to proving a `generic' crystalline result, we need a base change theorem for cohomology of the overconvergent site.

\begin{lemma} Suppose we have a Cartesian diagram 
\begin{equation}
\xymatrix{ X'\ar[d]_{f'}\ar[r]^{g'} & X\ar[d]^f \\ Y'\ar[r]^g & Y } 
\end{equation}
of $k$-varieties. Then for any sheaf $E\in (X/K)_{\an}$ the base change homomorphism 
\begin{equation}
g^*\R f_*E\rightarrow \R f'_* g'^*E
\end{equation}
is an isomorphism.
\end{lemma}

\begin{proof} Given the definitions, this is actually pretty formal, since we can view $(X/K)_{\an}$, $(Y'/K)_{\an}$ and $(X'/K)_{\an}$ as open subtopoi of $(Y/K)_{\an}$. However, we can also see it directly using realizations (and \S3.5 of \cite{LS11}) as follows. Let $(C,O)$ be an overconvergent variety over $(Y'/K)$. Then we have
\begin{align} (g^*\R f_* E)_{(C,O)} &= (\R f_*E)_{(C,O)} \\
 &= \R p_{X\times_Y C/O*} E|_{X\times_Y C/O} \\
 &=\R p_{X'\times_{Y'} C/O*} E|_{X'\times_{Y'} C/O} \\
 &=\R p_{X'\times_{Y'} C/O*} (g^*E)|_{X'\times_{Y'} C/O} \\
 &= (\R f'_* g'^*E)_{(C,O)}
 \end{align}
 as required.
\end{proof}

Hence, using Noetherian induction on $Y$, to prove that $\R f_*(\mathcal{O}_{X/K}^\dagger)$ is crystalline, it suffices to prove that it is generically crystalline, i.e. that there exists an open subset $U\subset Y$ such that $\R f_*(\mathcal{O}_{X/K}^\dagger)|_U$ is crystalline.

\begin{lemma}\label{flat}Suppose that $E\in D^+(\mathcal{O}_{Y/K}^\dagger)$ is a quasi-bounded above complex of $\mathcal{O}^\dagger_{Y/K}$-modules. If $\mathcal{H}^q(E)$ is a finitely presented crystal for all $q$, then $E$ is crystalline.\end{lemma}

\begin{proof} The key point is to show that the realizations of a finitely presented $\mathcal{O}_{Y/K}^\dagger$-module are flat. Indeed, once we know this, then, for any morphism $\mathbf{g}:(C',O')\rightarrow (C,O)$ of overconvergent varieties over $(Y/K)$, we know that 
\begin{equation}
\mathcal{H}^q(\mathbb{L}\mathbf{g}^\dagger_KE_{(C,O)})\cong \mathbf{g}^\dagger_K\mathcal{H}^q(E_{(C,O)})\cong \mathcal{H}^q(E_{(C',O')})
\end{equation}
and hence $\mathbb{L}\mathbf{g}_K^\dagger E_{(C,O)}\rightarrow E_{(C',O')}$ is a quasi-isomorphism.

Since crystals are of Zariski type, the question is local on $Y$, which we may therefore assume to be affine, and hence have a geometric realization $(Y,V)$. I first claim that for a finitely presented crystal $F$, $F_{(Y,V)}$ is a flat $i_Y^{-1}\mathcal{O}_V$-module. Let $F_0$ be the corresponding $j^\dagger\mathcal{O}_{]\overline{Y}[_{V_0}}$ module with overconvergent connection - this is locally free and is mapped to $i_{Y*}F_{(Y,V)}$ under the equivalence of categories 
\begin{equation}
\pi_{V*}:\mathrm{Coh}(j^\dagger\mathcal{O}_{]\overline{Y}[_{V_0}}) \cong \mathrm{Coh}(i_{Y*}i_Y^{-1}\mathcal{O}_V)
\end{equation}
which implies that the latter is flat. In general, we just note that locally any overconvergent variety $(C,O)$ over $Y/K$ admits a morphism to $(Y,V)$ and hence the result follows from the fact that the pull-back of a flat module is flat.  \end{proof}

\begin{proof}[of Proposition \ref{crys}] Just combine the previous lemmata.
\end{proof}

A certain amount of evidence for the `generic overconvergence' hypothesis of the proposition is given by the following translation of the main result of \cite{Shi08b} into the language of the overconvergent site.

\begin{proposition} Let $f:X\rightarrow Y$ be a morphism of $k$-varieties, which extends to a morphism of pairs $(X,\overline{X})\rightarrow (Y,\overline{Y})$ with $\overline{X}$ and $\overline{Y}$ proper. Then there exists an open subset $U\subset Y$ and a full subcategory $\mathcal{C}$ of triples over $(U,\overline{Y})$ satisfying the following condition. 

For any $q\geq0$ there exists a finitely presented crystal $E^q$ on $U$ such that for any $(Z,\overline{Z},\mathscr{Z})\in\mathcal{C}$ there is an isomorphism 
\begin{equation}
 \R^q f_*(\mathcal{O}^\dagger_{X/K})_{(Z,\mathscr{Z}_K)} \cong E^q_{(Z,\mathscr{Z}_K)}
\end{equation}
of $i_Z^{-1}\mathcal{O}_{\mathscr{Z}_K}$-modules, which functorial in $(Z,\overline{Z},\mathscr{Z})$.
\end{proposition} 

\begin{proof} Let $U,\mathcal{C},\mathcal{F}^q$ be as in Theorem 0.3 of \cite{Shi08b}. Let $E^q$ be the finitely presented $\mathcal{O}_{U/K}^\dagger$-module corresponding to the overconvergent isocrystal $\mathcal{F}^q$. Let $\pi:]\overline{Z}[_{\mathscr{Z}{0}}\rightarrow] \overline{Z}[_{\mathscr{Z}}$ denote the natural map. Since 
\begin{equation}
 \R^qf_{(X\times_Y Z,\overline{X}\times_{\overline{Z}}\overline{Y})/\mathscr{Z},\mathrm{rig}*}(\mathcal{O}_{X/K}^\dagger)\cong \mathcal{F}^q_{(Z,\overline{Z},\mathscr{Z})}
\end{equation}
is $j^\dagger\mathcal{O}_{]\overline{Z}[_{\mathscr{Z}0}}$-coherent, we know using Theorem \ref{extendedcomp} and the fact that $\pi_*$ is exact for coherent $j^\dagger\mathcal{O}_{]\overline{Z}[_{\mathscr{Z}0}}$-modules that 
\begin{equation}
\R^q f_*(\mathcal{O}_{X/K}^\dagger)_{(Z,\mathscr{Z}_K)}\cong  i_Z^{-1}\pi_*(\R^qf_{(X\times_Y Z,\overline{X}\times_{\overline{Z}}\overline{Y})/\mathscr{Z},\mathrm{rig}*}(\mathcal{O}_{X/K}^\dagger))
\end{equation}
where we are abusing notation slightly and writing $i_Z:]Z[_{\mathscr{Z}}\rightarrow ]\overline{Z}[_{\mathscr{Z}}$. Hence it suffices simply to note that $i_Z^{-1}\pi_*(\mathcal{F}^q_{(Z,\overline{Z},\mathscr{Z})})\cong E^q_{(Z,\mathscr{Z}_K)}$.
\end{proof}

\begin{remark} Of course, we have not said what the category $\mathcal{C}$ is, so the proposition as stated is not particularly useful. A full description of $\mathcal{C}$ comes from a precise statement of Shiho's result, which is Theorem 5.1 of \cite{Shi08b}. Another way to look at the proposition is that it is saying $\R^qf_*(\mathcal{O}^\dagger_{X/K})$ is generically a finitely presented crystal on some full subcategory of $(Y/K)_{\an}$. 
\end{remark}

\section{Rigid fundamental groups and homotopy obstructions}

In the previous sections, we have defined absolute and relative rigid rational homotopy types. These are dga's, and we can apply the bar construction to obtain algebraic models of path spaces. Thus we can extract pro-unipotent groups which in some sense deserve to be called unipotent fundamental groups. However, there are already definitions of these - in the absolute case we have the Tannaka dual of the category of unipotent isocrystals, and in the relative (smooth and proper) case, there is a definition of the unipotent fundamental group given in \cite{Laz15}. One would like to compare these constructions and show that they give the same answer, and in this section we do so in the absolute case.

Here, we can basically copy Olsson's proof for convergent homotopy types of smooth and proper varieties. Recall that we have functors 
\begin{align}
D:\ho{\mathrm{dga}_K}& \rightarrow \ho{\mathrm{Alg}_K^{\Delta}} \\
\R\mathrm{Spec}:\ho{\mathrm{Alg}_K^\Delta}^\circ&\rightarrow \ho{\mathrm{SPr}(K)}
\end{align}
and Olsson has shown in his preprint \cite{Ols} that the bar construction $\pi_1$ of a dga $A$ coincides with the topological $\pi_1$ of the simplicial presheaf $\R\spec{D(A)}$. Hence, it suffices to prove the comparison between this topological $\pi_1$ of the rational homotopy type \begin{equation}
(X/K)_\mathrm{rig}:=\R\spec{D(\R\Gamma_\mathrm{Th}(\Omega^*(\mathcal{O}_{X/K}^\dagger)))}
\end{equation}
and the Tannakian $\pi_1$ of $X/K$.

In the smooth and proper case, working with the convergent site, this is proved by Olsson in \S2 of \cite{Ols07b}, and his proof adapts fairly easily to the rigid case. Rather than writing out the whole proof in our slightly different situation, we will just make a few comments that we hope will convince the reader that the necessary changes are easily made.

Owing to the comparison results both of \S\ref{overcon} above and of Le Stum's paper \cite{LS11}, we can everywhere in the construction of $\R\Gamma_{\mathrm{Th}}(\Omega^*(\mathcal{O}^\dagger_{X/K}))$ replace rigid spaces by Berkovich spaces. We can also easily construct the `cohomology complexes' of ind-coherent crystals of $\mathcal{O}^\dagger_{X/K}$-modules on the overconvergent site, exactly as in \S2.24 of \cite{Ols07b} by taking framing systems and realizations on these framing systems. This allows us to define the pointed stack $(\widetilde{X}/K)_\mathrm{rig}$ analogously to \S2.29 of \emph{loc. cit.}, but instead taking $\widetilde{G}$ to be the pro-unipotent Tannakian fundamental group rather than the whole pro-algebraic fundamental group. (Note that in our case, because we are only working with unipotent isocrystals,  $G=1$). 

The proof of Proposition 2.35 and Lemma 2.36 needs to be slightly modified as follows. Let $\pi$ denote the functor of $\widetilde{G}$-invariants (of sheaves or modules), and let $C^\bu(-)$ denote the cohomology complex of an ind-coherent crystal of $\mathcal{O}^\dagger_{X/K}$-modules. Let $\mathbb{L}(\mathcal{O}_{\widetilde{G}})$ be the overconvergent version of Olsson' object of the same name. Then as in Proposition 2.35 we need to compare $\R\Gamma_\mathrm{rig}(\mathcal{V})$ and $\R\pi(\R\Gamma_\mathrm{an}(C^\bu(\mathcal{V}\otimes\mathbb{L}(\mathcal{O}_{\widetilde{G}}))))$ for a unipotent overconvergent isocrystal $\mathcal{V}$, which is equivalent to comparing $\R\Gamma_\mathrm{rig}(\mathcal{V})$ and $\R\pi(\R\Gamma_\mathrm{rig}(\mathcal{V}\otimes\mathbb{L}(\mathcal{O}_{\widetilde{G}})))$. Since $\pi$ and $\Gamma_\mathrm{rig}$ commute, as in the proof of Lemma 2.36 it suffices to show that $\mathcal{V}\cong \R\pi(\mathcal{V}\otimes \mathbb{L}(\mathcal{O}_{\widetilde{G}}))$, and the proof of this follows exactly as in \emph{loc. cit.}, using the overconvergent rather than the convergent site. To summarise we have the following theorem.

\begin{theorem} The Tannakian unipotent fundamental group of a $k$-variety $X$ at a point $x\in X(k)$ coincides with the unipotent fundamental group obtained from the augmented dga $\R\Gamma_\mathrm{Th}(\Omega^*(\mathcal{O}_{X/K}^\dagger))$ via the bar construction. In particular, if $k$ is a finite field, then the linear Frobenius structure on the (co-ordinate ring of the) former is mixed.
\end{theorem}

\begin{remark} Mixed structures on the unipotent Tannakian fundamental group have already been studied by Chiarellotto in \cite{Chi98}, where he defines a weight filtration on the completed universal enveloping algebra of the Lie algebra of the unipotent Tannakian fundamental group.
\end{remark}

\begin{remark} Unfortunately, at the moment this results seems difficult to extend to the relative case, for multiple reasons, of which we will not go into detail here.
\end{remark}

A reason that we are interested in this comparison is that in \cite{Laz15} we defined a function field analogue of Kim's non-abelian period map \begin{equation}
X(S)\rightarrow H^1_{F,\mathrm{rig}}(S,\pi^\mathrm{rig}_1(X/S,p))
\end{equation}
which takes sections of a smooth and proper scheme $f:X\rightarrow S$ over a curve over $k$ to a certain set classifying $F$-torsors under the relative unipotent fundamental group, at some base point $p\in X(S)$.

Basic functoriality of relative rational homotopy types in this situation gives a map
\begin{equation}
X(S)\rightarrow [\R f_*(\mathcal{O}_{X/K}^\dagger),\mathcal{O}^\dagger_{S/K} ]_{F\text{-}\ho{\mathrm{dga}(\mathcal{O}_{S/K}^\dagger)}} 
\end{equation} 
where the RHS is maps in the homotopy category, and we would like to compare these two period maps. To do so, we will certainly need to compare the Tannakian construction of the relative fundamental group with the relative rational homotopy type.

\subsection{A rather silly example}
Recall that when discussing homotopy obstructions in the absolute case, we noted that the non-existence of a section of the map $f:\A^1_k\rightarrow \A^1_k$, $x\mapsto x^2$ could not be detected on the level of rational homotopy, because the rational homotopy type of $\A^1_k$ is trivial. However, we can see this non-existence on the level of relative rational homotopy types. Indeed, it clearly suffices to show that there is no section of the $x\mapsto x^2$ map on $\A^1_k\setminus\{0\}$, and here we can explicitly describe the (isomorphism class of the) push-forward of the constant isocrystal $f_*(\mathcal{O}^\dagger_{\A^1_k\setminus\{0\}/K})$. It is a free rank 2 module over $K\langle t,t^{-1}\rangle^\dagger$, with connection and algebra structures defined by 
\begin{equation}
\nabla\left( \begin{matrix}f \\ g\end{matrix}\right)= \left( \begin{matrix} df  \\dg-g\frac{dt}{2t}\end{matrix} \right),\quad \left( \begin{matrix}f_1 \\ g_1\end{matrix}\right)\left( \begin{matrix}f_2 \\ g_2\end{matrix}\right)=\left( \begin{matrix}f_1f_2+g_1g_2 \\ f_1g_2+f_2g_1\end{matrix}\right).
\end{equation} 
It is simple to verify that there cannot be a morphism $ f_*(\mathcal{O}^\dagger_{\A^1_k\setminus\{0\}/K})\rightarrow \mathcal{O}^\dagger_{\A^1_k\setminus\{0\}/K}$ compatible with both the algebra structures and the connection, and hence that there can be no section of $f$ on $\A^1_k\setminus\{0\}$.

Of course this example is rather stupid - one does not need the huge machinery of homotopy theory and the overconvergent site to show that there is no square root of $t$ in $k[t]$! However, this example is instructive for two reasons.

\begin{itemize}
\item It shows that the relative rational homotopy type contains strictly more information that just looking at the map between the absolute rational homotopy types. 
\item The algebra structure was crucial in showing the non-existence of a section of homotopy types - there certainly is a section of the cohomology, but it is not multiplicative.
\end{itemize}

\section*{Acknowledgements}
This paper is intended to form part of the author's PhD thesis, and was written under the supervision of Ambrus P\'{a}l at Imperial College, London. He would like to thank Dr P\'{a}l both for suggesting the direction of research and for providing support, encouragement and many hours of fruitful discussions. He would also like to thank the anonymous referee for a careful reading of the paper, for pointing out errors and suggesting improvements.

\bibliographystyle{mysty}
\bibliography{/Users/Chris/Dropbox/LaTeX/lib.bib}
	
\end{document}